\documentclass[12pt]{amsart}
\usepackage{amsmath}
\usepackage{amsthm}
\usepackage{amsfonts} 
\usepackage{hyperref}
\usepackage{epsfig}
\usepackage{amssymb}
\usepackage{dsfont}
\usepackage{graphicx}
\usepackage{subfigure}
\usepackage{color}

\numberwithin{equation}{section}

\allowdisplaybreaks

\newcommand{\cred}{\color{red}}

\newtheorem{prop}{Proposition}[section]
\newtheorem{rem}[prop]{Remark}
\newtheorem*{lem*}{Lemma}
\newtheorem{lem}[prop]{Lemma}
\newtheorem{theo}[prop]{Theorem}
\newtheorem{coro}[prop]{Corollary}

\newtheorem{definition}[prop]{Definition}
\newtheorem*{claim*}{Claim}
\newtheorem*{remerciements*}{Acknowledgements}

\newcommand{\Cst}{\mathsf{Cst}}
\newcommand{\e}{\epsilon}

\newcommand{\beq}{\begin{eqnarray}}
\newcommand{\beqq}{\begin{eqnarray*}}
\newcommand{\eeq}{\end{eqnarray}}
\newcommand{\eeqq}{\end{eqnarray*}}
\newcommand{\eps}{\varepsilon}
\newcommand{\Pb}{\mathbb{P}}
\newcommand{\G}{\Gamma}
\newcommand{\N}{\mathbb{N}}

\newcommand{\A}{\mathcal{A}}
\newcommand{\Fn}{\mathcal{F}_n}
\newcommand{\E}{\mathbb{E}}
\newcommand{\V}{\mathbb{V}}
\newcommand{\De}{\Delta}

\newcommand{\simx}{\stackrel{x}{\sim}}
\newcommand{\simG}{\stackrel{\mathcal{G}}{\sim}}

\newcommand{\1}{\mathds{1}}
\newcommand{\Pc}{\mathcal{P}}
\newcommand{\Gc}{\mathcal{G}}
\newcommand{\de}{\delta}
\newcommand{\Es}{\mathbb{E}}
\newcommand{\F}{\mathcal{F}}
\newcommand{\iy}{\infty}

\renewcommand{\le}{\leqslant}
\renewcommand{\leq}{\leqslant}
\renewcommand{\ge}{\geqslant}

\title{Reinforcement learning in social networks}
\date{}
\author[D.~Kious]{Daniel Kious}
\address{Daniel Kious\\Ecole Polytechnique F\'ed\'erale de Lausanne\\ EPFL SB MATHAA PRST\\
MA B1 537, Station 8\\
CH-1015 Lausanne, Switzerland} \email{daniel.kious@epfl.ch}
\author[P.~Tarr\`es]{Pierre Tarr\`es}
\address{Pierre Tarr\`es, CNRS and Université Paris-Dauphine, PSL Research University, Ceremade (UMR 7534),
Place de Lattre de Tassigny, 75775 Paris Cedex 16, France} \email{tarres@ceremade.dauphine.fr}

\begin{document}
\maketitle


\begin{abstract}
We propose a model of network formation based on reinforcement learning, which can be seen as a generalization as the one proposed by Skyrms \cite{APSV,HST} for signaling games. On a discrete graph, whose vertices represent individuals, at any time step each of them picks one of its neighbors with a probability proportional to their past number of communications; independently, Nature chooses, with an independent identical distribution in time, which ones are allowed to communicate. Communications occur when any two neighbors mutually pick each other and are both allowed by Nature to communicate.   

Our results generalize the ones obtained by Hu, Skyrms and Tarr\`es in \cite{HST}. We prove that, up to an error term, the expected rate of communications increases in average, and thus a.s.~converges. If we define the limit graph as the non-oriented subgraph on which edges are pairs of vertices communicating with a positive asymptotic rate, then, for stable configurations, within which every vertex is connected to at least another one, the connected components of this limit graph are star-shaped and satisfy a certain balance condition. Conversely, given any stable equilibrium $q$ whose associated graph satisfies that property, the occupation measure converges with positive probability to a stable equilibrium in a neighborhood of $q$ with the same limit graph.
\end{abstract}


\section{Introduction}\label{sect_intro}


We introduce and analyse a model of network formation, based on one hand on a reciprocity condition - we can talk to somebody only if he conversely wants to talk to us -  and on the other hand on a reinforcement learning procedure - we want to talk more to the ones we already talked to frequently.

Start with a weighted graph $G=(\V,E,\sim)$, where the vertices in $\V$ represent people, and nonoriented edges in $E$ represent links between them. 

Let $\A:=(a_{ij})_{i,j\in\V}$ be a collection of nonnegative real numbers such that, for any $i,j\in\V$, $a_{ij}=a_{ji}$, and if $a_{ij}>0$ then $\{i,j\}\in E$; $a_{ij}$ represents the affinity between $i$ and $j$. Let the network $G_\A$ be the graph $G$ with weights $\A$.

Now consider the following game on the network $G_{\mathcal{A}}$, whose players are the vertices of $G$. 
The game is played in infinitely many rounds (one at each time step), and each of them consists in the following procedure:
\begin{itemize}
\item each vertex $i\in \mathbb{V}$ chooses one, and only one, of its neighbours $j\sim i$;
\item Nature independently picks a subset $V\subseteq\mathbb{V}$ of vertices, allowed to communicate;
\item if $i$, $j$ $\in V$ are neighbours, i.e.~$i\sim j$, and if $i$ and $j$ mutually choose each other, then a communication occurs between them and they both receive a payoff equal to their affinity $a_{ij}$.
\end{itemize}

We model the choice by Nature of the set of vertices which are allowed to communicate at each time step by random independent identical Bernoulli distributions on the set of subsets of $\mathbb{V}$, denoted by $\Pc(\mathbb{V})$. 

More precisely, we let $(p_V)_{V\in\Pc(\mathbb{V})}$ be a family of nonnegative real numbers such that
$$\sum_{V\in\Pc(\mathbb{V})} p_V=1,$$
and assume that the sequence of subsets $V_n\in\Pc(\mathbb{V})$ chosen by Nature at time $n\in\N$ is an i.i.d. sequence of Bernoulli random variables with probability $(p_V)_{V\in\Pc(\mathbb{V})}$.

Note that, given any two adjacent vertices  $i,j\in\V$, $i\sim j$, 
$$p_{ij}:= \sum_{V\in\Pc(\mathbb{V}):\,i,j\in V}p_V$$
is the probability that $i$ and $j$ are both allowed to communicate by Nature at any time step. By convention, we fix $p_{ij}=0$ when $i\nsim j$. The numbers $a_{ij}p_{ij}$, $i\sim j$, will be important in the results.  In general  the nonoriented edges $\{i,j\}$ are not chosen independently of each other by Nature.  

Let us describe three noticeable choices of $(p_V)_{V\in\Pc(\mathbb{V})}$. 

The most obvious one is $p_\V=1$, and $p_V=0$ if $V\neq\V$, where any adjacent pair of vertices is allowed to communicate together at any time.

Another natural choice is that of a sequence $(p_V)_{V\in\Pc(\mathbb{V})}$, such that $p_V\ne0$ if and only if (iff) $V=\{i\}\cup\{j:j\sim i\}$ for some $i\in\V$: at each time step only one particular (random) vertex $i$ and its neighbours are allowed  by Nature to communicate together.

The third choice is that of a signaling game, on which the same reinforcement learning procedure was studied in \cite{HST}. Let $\mathcal{S}_1$ and $\mathcal{S}_2$ be two disjoint subsets spanning $\V$, i.e. $\V=\mathcal{S}_1\cup\mathcal{S}_2$, $\mathcal{S}_1\cap\mathcal{S}_2=\emptyset$. Then assume $p_{V}\ne0$ iff $V=V_i:=\{i\}\cup\mathcal{S}_2$ for some $i\in\mathcal{S}_1$. Then $p_{ij}>0$ only if $i\in\mathcal{S}_p$, $j\in\mathcal{S}_q$, $p$, $q$ $\in\{1,2\}$, $p\ne q$. In other  words the graph on which communications occur is bipartite in this case, see Figure \ref{bipartite}.
In \cite{HST}, the bipartite graph is complete, and it is assumed that $p_{V_i}=1/M_1$, where $M_1$ is the number of vertices in $\mathcal{S}_1$, but the condition is not required in the current paper.
\begin{figure}[h]
\center
\includegraphics{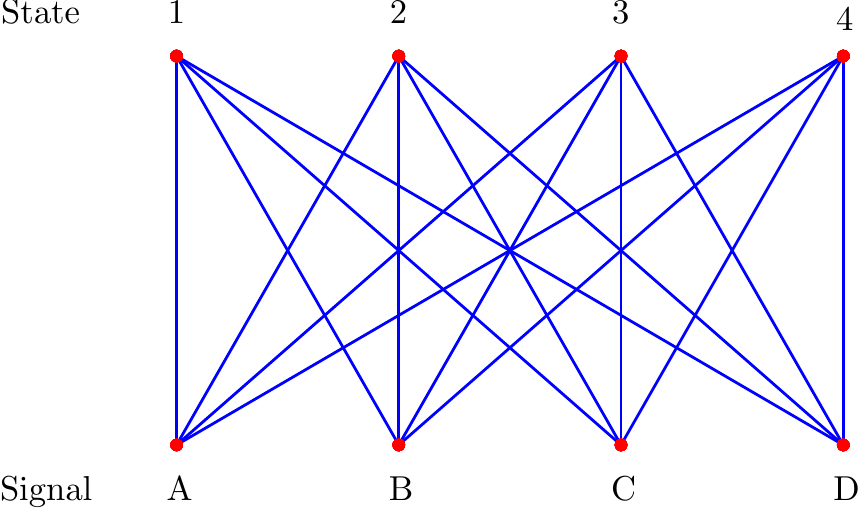}
   \caption{\label{bipartite} In the signaling game, $\mathcal{S}_1$ and $\mathcal{S}_2$ correspond respectively to the sets of states of Nature and signals: the agents learn to signal, i.e. they create a common language.}
\end{figure}

The game being defined, let us now describe how the individuals pick each other. The model we consider here is reinforcement learning, a trial-and-error procedure in which agents are more likely to use the strategies that have given more payoff before, which in our case means that, \emph{the more an agent has already talked to somebody, the more he will be likely to again  talk to him}. At each round, each vertex will choose one of its neighbours with a probability that is proportional to the payoff they shared together so far.

More precisely, for all $i,j\in\V$, let us define the \emph{cumulative payoff} $V_{ij}^n$ (resp. $V_{i}^n$) on the edge $\{i,j\}$ (resp. vertex $i$) at time $n$ by 
\begin{align*}
V_{ij}^n&=v_{ij}^0+a_{ij}N_{ij}^n,\\
V_i^n&=\sum_{j\in\V:j\sim i} V_{ij}^n,
\end{align*} 
where $v^0_{ij}$ is a nonnegative constant such that $v^0_{ij}>0$ iff $i\sim j$, and $N_{ij}^n$ is the number of communications that succeeded between $i$ and $j$ at time $n$. If $i\nsim j$, we set $V_{ij}^n=0$ for all $n$. Note that we could have $i\sim j$ (so that $v_{ij}^0>0$) but $a_{ij}=0$, in which case $V_{ij}^n=v_{ij}^0$ for all $n$.

Now we assume that, at each time step, any vertex $i\in\V$ chooses one of its neighbours $j$ with probability
$$V_{ij}^n/V_i^n.$$
By symmetry, $j$ also chooses $i$ among its neighbours with probability
$$V_{ij}^n/V_j^n.$$
Let $\mathcal{F}:=(\Fn)_n$ be the filtration generated by the process, i.e $$\Fn:=\sigma((V_{ij}^0,...,V_{ij}^n),i,j\in\V), \,\,\,n\in\mathbb{N}.$$
Then, the probability of a communication through an edge $ij\in E$ is
\beq\label{proba_succes}
\Pb\left.\left(V_{ij}^{n+1}=V_{ij}^n+a_{ij}\right|\Fn\right)=p_{ij}\frac{(V_{ij}^n)^2}{V_i^nV_j^n},
\eeq
and the communication is considered to be successful if $V_{ij}^{n+1}=V_{ij}^n+a_{ij}$ and $a_{ij}>0$. 

We have now  described our learning procedure of the network game.
Observe that reinforcement learning is one of a number of models of strategic learning in games, in which players adapt their strategies, with the (possibly unconscious) aim to eventually maximize their payoffs, amongst for instance no-regret learning, fictitious play and its variants, and hypothesis testing, see \cite{young:2005}. 

In the spectrum of adaptive procedures, reinforcement learning certainly does not provide an optimal strategy, and rather assumes that players have bounded rationality. In particular they do not need to know which game the other agents are playing, or which strategies they are playing. It is therefore attractive as a simple behavioral model, since it does not require the agents to be  entirely devoted to their task and, in our setting, even to be aware that they are involved in the formation of a network, but on the contrary just to observe their effective payoffs. Moreover it accumulates inertia, since the relative increase in payoff decreases in time, which would make it more stable with respect to randomness of the payoffs for instance, and in general avoid instability of the strategies. 

Models of network formation using reinforcement learning were already proposed by Pemantle and Skyrms~\cite{RPBS} as follows: each day, each individual chooses one of his neighbours and talks to him, the communication being always accepted by the neighbour; then both of them reinforce the probability to talk to each other. An important difference is that, in our model, a reciprocity assumption is made, in the sense that the communication is not necessarily accepted by the chosen neighbour.\\
A similar model, with strong reinforcement, has been proposed by Van der Hofstad {\it et al.~}in the context neural network formation, see \cite{vdH}.\\

Let us now summarize the results obtained in this paper, which are described in more detail in Section \ref{sect_th}. As explained above, the network game model can be seen as a generalisation of the signaling game \cite{Sky10}, on which the same reinforcement learning procedure was analyzed by Hu, Skyrms and Tarr\`es~\cite{HST}. Our results are generalisations of the ones in \cite{HST}, and the proofs adapt their techniques to a more general setting.

The main tools are stochastic approximation techniques. In particular, we will approach the behavior of some stochastic process by a deterministic ordinary differential equation (ODE), with the difficulty that the function governing the ODE is not continuous on the boundary of the simplex on which it is defined.  One could instead show that the process, with a different system of coordinates, approximates a smooth dynamics, but it would then take place in an unbounded domain, which would lead to other technical difficulties (see Section~\ref{sect_bord}).   

First, we prove  that there is a Lyapunov function for the deterministic ODE, which enables us to deduce in Theorem~\ref{th_pot} the a.s.~convergence of the expected payoff of the random dynamics. Then we prove in Theorem~\ref{th_eq} that the occupation measure of the communications converges a.s.~to the set of equilibria of the ODE.

In Proposition \ref{Prop:unstable0}, we show a property which characterizes the stable equilibria of the ODE: if we define the limit graph as the non-oriented subgraph on which edges are pairs of vertices communicating successfully infinitely often, the connected components of this limit graph are star-shaped  components, in other words \emph{core and shell} (see Figure~\ref{stab_conf}), satisfying furthermore a balance condition on the affinities $a_{ij}$ and probabilities $p_{ij}$. In particular each connected component contains a nucleus vertex linked to (one or several) satellite vertices which are only linked to the nucleus. 

In Theorem~\ref{th_posprob} we prove that any graph correspondence with the preceding property, such that no vertex is falling out of use (see Section \ref{sect_th} for a precise definition), is a limit configuration with positive probability.

\section{Main results}\label{sect_th}
We let  $\Cst(a_1,a_2,\ldots, a_p)$ denote a positive constant depending
only on  $a_1$, $a_2$, $\ldots$ $a_p$, and let $\Cst$  denote a  universal
positive constant.

Given two real functions $f$ and $g$ defined on a set $D$, we write $f=O(g)$ iff there exists a constant $C>0$ such that $|f(x)|\le C g(x)$ for all $x\in D$.

First, notice that if there is no $i,j\in \V$ such that $a_{ij}p_{ij}>0$, then we are the trivial case where no successful communication can ever occur hence the system is just frozen in its initial configuration. Therefore, we assume that there exist $i,j\in \V$ such that $a_{ij}p_{ij}>0$ for the rest of the paper.\\

For all $n\in\mathbb{N}$, define:
\beq
T_n&:=&\sum_{i,j\in\V} V_{ij}^n,\label{def_Tn}\\
x_{ij}^n&:=&\frac{V_{ij}^n}{T_n}\text{, for all }i,j\in\V,\nonumber\\
x_{i}^n&:=&\frac{V_{i}^n}{T_n}=\sum_{j\in\V} x_{ij}^n\text{, for all }i\in\V.\nonumber
\eeq
Let $x_n:=(x_{ij}^n)_{i,j\in\V}$ be the \emph{occupation measure} at time $n$. Note that $\sum_{i,j}x_{ij}^n=1$. 
Moreover, it is straightforward to see that for all $i,j\in \V$, $j\sim i$, with $a_{ij}p_{ij}=0$, we have $V^n_{ij}=v^0_{ij}$ a.s.~for all $n\in\mathbb{N}$. Hence, for any $n\in \mathbb{N}$,
\begin{align}\label{def_h1}
\sum_{i,j: a_{ij}p_{ij}>0} {x^n_{ij}}=1-\sum_{i,j: a_{ij}p_{ij}=0} \frac{v^0_{ij}}{T_n}
\ge1-\sum_{i,j: a_{ij}p_{ij}=0} \frac{v^0_{ij}}{T_0}=\sum_{i,j: a_{ij}p_{ij}>0} {x^0_{ij}}=:h_1\text{ a.s.}
\end{align}
Moreover,  $h_1>0$ as we assumed that there exists some $i,j\in\V$ such that $a_{ij}p_{ij}>0$, which implies $v_{ij}^0>0$.\\
Therefore,  the occupation measure $x_n$ belongs to the simplex:
\begin{align}\label{def_delta}
\Delta:=\Big\{(x_{ij})_{i,j\in\V}: \sum_{i,j\in\V}x_{ij}=1\text{ and }\sum_{i,j: a_{ij}p_{ij}>0} {x_{ij}}\ge h_1,\\ \nonumber
\text{where } x_{ij}=x_{ji}\ge0\text{, and } x_{ij}=0 \text{ if } i\nsim j\Big\}.
\end{align}

\begin{rem}
Let us comment the restriction $\sum_{i,j: a_{ij}p_{ij}>0} {x_{ij}}\ge h_1$ in the definition of $\Delta$. This will be useful in the analysis of the equilibria and their stability done in Section \ref{sect_eq}. Indeed, without this assumption, we could find some pathological equilibria which are of no interest for the random process we study.
\end{rem}

In order to define the \emph{expected payoff}, let us compute the conditional increment of $T_n$, for any $n\in\mathbb{N}$, recalling \eqref{proba_succes},
\beq\label{incr_T}
\E\left(\left.T_{n+1}-T_n\right|\mathcal{F}_n\right)&=& \sum_{i,j\in\V} \E\left(\left.V_{ij}^{n+1}-V_{ij}^n\right|\mathcal{F}_n\right)=\sum_{i,j\in\V} a_{ij}p_{ij}\frac{(x_{ij}^n)^2}{x_i^nx_j^n}.
\eeq

\begin{definition}
Let $H:\mathbb{R}_+^{V\times V}\longrightarrow\mathbb{R}_+$ be the function defined, for all $x\in\Delta$, by
\[
H(x):=\sum_{i,j\in\V:x_{ij}>0}a_{ij}p_{ij}\frac{x_{ij}^2}{x_ix_j}.
\]
If $x\in\De$, we call $H(x)$ the expected payoff at $x$.
\end{definition}

In this paper we use stochastic approximation techniques, namely we compare the evolution of the random process $(x_n)$ to the behavior of the deterministic dynamics driven by the mean-field ODE
\beq\label{ode}
\frac{dx}{dt}=F(x),
\eeq
where $F$ is a function from $\Delta$ to $T\Delta$, tangent space of $\Delta$, which maps $x$ to
\[
\label{deff}
F(x)=\Bigg[x_{ij}\big(a_{ij}p_{ij}\frac{x_{ij}}{x_ix_j}-H(x)\big)\Bigg]_{i,j\in\V},
\]
with the convention that $F(x)_{ij}=0$ if $x_{ij}=0$ and that $a_{ij}p_{ij}x_{ij}^2/x_ix_j=0$ whenever $a_{ij}p_{ij}=0$.
We will make the link between the ODE and the random process explicit in Section~\ref{sect_ode}.

 As we will see, $H$ is a Lyapunov function for the ODE~\eqref{ode}, which should imply that, up to a small error term, the random process $(H(x_n))_n$ increases in average. We are indeed able to show the convergence of $(H(x_n))_n$, and hence the asymptotic linear growth of $(T_n)_n$ by conditional Borel-Cantelli Lemma, see Corollary~\ref{cor_Tn}. The proof  is technical, since the function $H$ is irregular on the boundary of the simplex.

\begin{theo}\label{th_pot}
The expected payoff process $(H(x_n))_{n\in\N}$ converges almost surely. 
\end{theo}

The next Theorem \ref{th_eq} shows the convergence of $(x_n)_{n\in\mathbb{N}}$ towards the {\it set of equilibria} of the ODE~\eqref{ode}, defined by
\beq\label{def_Gamma}
\Gamma:=\Big\{x\in\Delta\,:\,F(x)=0\Big\}.
\eeq
The equilibria are not isolated in general, and this result does not imply the a.s. convergence of  $(x_n)_{n\in\N}$, which we could not prove in general.\\
We call $x\in\De\setminus\partial\De$ a {\it stable equilibrium} of the ODE ~\eqref{ode} iff $x\in\Gamma$ and the maximum real part of the eigenvalues of the Jacobian matrix of $F$ at $x$ is nonpositive, see Definition~\ref{defstab}. This condition is close to the following one: if the solution of the  ODE~\eqref{ode} starts near $x$, it will remain in its neighbourhood forever (see Section~\ref{sect_eq} for more detail).

\begin{theo}\label{th_eq}
The random processes $(F(x_n))_{n\in\N}$ and  $(x_n)_{n\in\mathbb{N}}$ respectively converge a.s.~to $0$ and to the set of equilibria $\Gamma$.
\end{theo}

Our next result provides necessary and sufficient conditions for the stability of equilibria of the ODE. 

Let us first define the \emph{boundary} of the simplex as
\beqq
\partial\Delta:=\left\{x\in\Delta:\exists i\in\V \text{ s.t. } \sum_{j\in\V:a_{ij}p_{ij}>0}x_{ij}=0\right\}.
\eeqq

Note that $\partial\Delta$ is not the topological boundary of $\Delta$. If $x_n\to\partial\De$, then one of the vertices \emph{falls out of use}, in the sense that the frequency of its communications asymptotically goes to zero. Let us emphasize that we could have some vertex $i$ falling out of use without making $x_n$ converging to $\partial\Delta$: this happens if $a_{ij}p_{ij}=0$ for all $j\sim i$.

We characterize in Proposition \ref{Prop:unstable0} the stable equilibria $x\in\De\setminus\partial\De$ in terms of a graph structure associated to $x$. 

Let us first introduce some definitions. For any $x\in\Delta$, we define a subgraph $G_x$ of $G$, with adjacency $\simx$, of possible  communications between vertices associated to $x$. Then Definition \ref{def_prop} will introduce the property corresponding to the stability of  equilibria in $\Delta\setminus\partial\Delta$.

 \begin{definition} \label{def_Gx}
Given $x\in \Delta$, let $G_x$ be the subgraph of $G$ with vertices in
$\V$ and adjacency $\simx$ such that $i\simx j$ if and only if $x_{ij}>0$, for all $i,j\in\V$.
\end{definition}
\begin{definition} \label{def_prop}Consider a subgraph $\mathcal{G}$ of $G$, with adjacency $\simG$, and let $\mathcal{C}_1,...,\mathcal{C}_d$ be its connected components. Let $P_\mathcal{G}$ be the following property:
\begin{enumerate}
 \item $\forall m\in\{1,...,d\}$, $i,j,k,l\in\mathcal{C}_m$, s.t. $i\simG j$ and $k\simG l$, $a_{ij}p_{ij}=a_{kl}p_{kl}>0$;
\item $\forall m\in\{1,...,d\}$, $\mathcal{C}_m$ contains at most one vertex with several neighbours;
\item a vertex $i\in \V$  has a corresponding edge within $\mathcal{G}$ if and only if $a_{ij}p_{ij}>0$ for some $j\sim i$.
\end{enumerate}
\end{definition}

\begin{definition}\label{def_nucleus}
Consider a subgraph $\mathcal{G}$ of $G$ which satisfies $P_\mathcal{G}$. For each connected component $\mathcal{C}$ of $\mathcal{G}$, we define the {\it nucleus vertex} of $\mathcal{C}$ as the single vertex $i\in\mathcal{C}$ which has several neighbours, chosen arbitrarily when $\mathcal{C}$ contains exactly two vertices, or as the single vertex in $\mathcal{C}$ when this component consists of an isolated vertex.
\end{definition}

Condition $(1)$ in Definition \ref{def_Gx} is a balance condition on the affinities within a connected component, whereas Condition $(2)$ means that each connected component is star-shaped: if $P_{\mathcal{G}}$ holds then within each connected component, there exists a \emph{nucleus vertex} $i_0$ such that if $j_1\simG i_0$ is a \emph{satellite vertex}, then $j_1\stackrel{\mathcal{G}}{\nsim} i$ for any $i\neq i_0$, see Figure~\ref{stab_conf}.
\begin{figure}[h]
\center
   \includegraphics{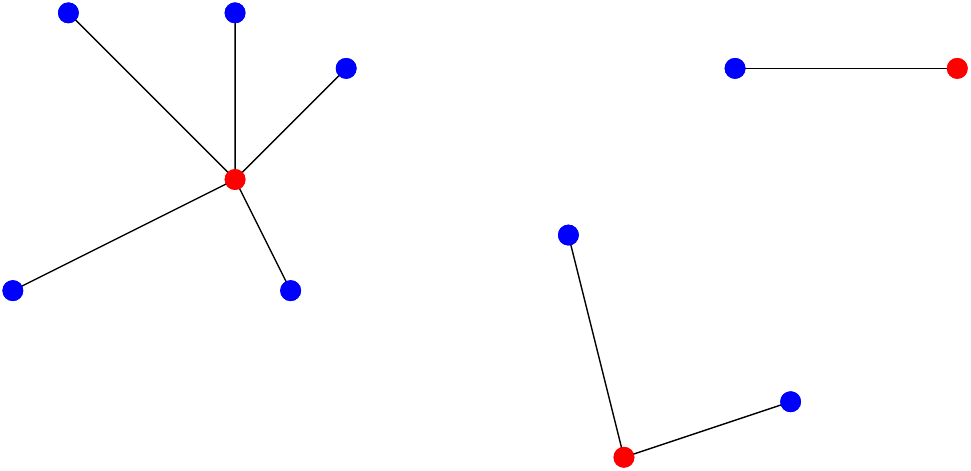}
   \caption{\label{stab_conf} Stable configuration composed of three star-shaped connected components}
\end{figure}
Finally, Condition $(3)$ applied to $\mathcal{G}=G_x$ is equivalent to $x\in\Delta\setminus\partial\Delta$.

\begin{definition}\label{def_gammag}
Let $\Gc$ be a subgraph of $G$ such that $P_\Gc$ holds and let us denote $N$ the set of its nucleus vertices.   We define the set $\Gamma_\Gc$ such that $q\in\Gamma_\Gc$ if
\begin{enumerate}
\item $G_q=\Gc$ and consequently $q_{ij}=0\Leftrightarrow i\stackrel{\Gc}{\nsim} j$;
\item for any $i\in N$, $q_i=(ap)_i/(2\sum_{j\in N} (ap)_j)$ where, for any $j\in N$,  $(ap)_j=a_{jk}p_{jk}$ for some/any $k\simG j$;
\item for any $i\in N$,  $(q_{ij})_{j:j\simG i}$ is any collection of positive numbers such that $\sum_{j:j\simG i} q_{ij}=q_i$.
\end{enumerate}
In particular, $\Gamma_\Gc\subset\Delta\setminus\partial\Delta$ and $\Gamma_\Gc\neq \emptyset$.
\end{definition}

The following result is a consequence of Proposition~\ref{Prop:unstable} and Proposition \ref{prop:eqg}.

\begin{prop}
\label{Prop:unstable0}
Let $x\in\Delta\setminus\partial\Delta$. Then, the following assertions are equivalent:
\begin{enumerate}
\item[$(i)$] $x$ is a stable equilibrium;
\item[$(ii)$] $x$ is an equilibrium and its associated subgraph $G_x$ satisfies the property $P_{G_x}$;
\item[$(iii)$] $x\in \Gamma_{\Gc}$, for some subgraph $\Gc$ satisfying $P_\Gc$.
\end{enumerate}
\end{prop}
\begin{rem}
This last result implies that, if $P_{\mathcal{G}}$ holds for some subgraph $\Gc$, then the set of stable equilibria $x$ with $G_x=\Gc$ is a non-empty continuum .
\end{rem}

Theorem \ref{th_posprob} implies that any subgraph $\mathcal{G}$ of $G$ such that $P_\Gc$ holds, with positive probability $x_n$ converges to some stable equilibrium $x\in\Delta\setminus\partial\Delta$ with $G_x=\Gc$; and $\Gc$ is a limit graph in a strong sense, since after a (random) time communications will only occur within edges of this graph.
\begin{theo}\label{th_posprob}
Let $\Gc$ be a subgraph of $G$ such that $P_\Gc$ holds. Fix $q\in\Gamma_{\Gc}$, so that, in particular, $q\in\Delta\setminus\partial\Delta$, $G_q=\Gc$ and $q$ is a stable equilibrium for the ODE~\eqref{ode} associated to the evolution of $(x_n)$.\\
Let $\mathcal{N}(q)$ be a neighbourhood of $q$ in $\Gamma_{\Gc}$. Then, with positive probability,
\begin{enumerate}
 \item $x_n\to x\in\mathcal{N}(q)$, where $x$ is thus a stable equilibrium with $G_x=\Gc$.
\item $\forall i,j\in\V$, $V^\infty_{ij}=\infty\,\iff\,\{i,j\}$ is an edge of $\Gc$.
\end{enumerate}
\end{theo}

Theorems~\ref{th_pot}, \ref{th_eq} and \ref{th_posprob}, and Proposition \ref{Prop:unstable0}, are generalisations of results of Hu, Skyrms and Tarr\`es \cite{HST}, in the case of the so-called signaling game, see Section~\ref{sect_intro}.\\

We believe but cannot prove that $x_n$ converges a.s.~to some stable equilibrium in $\Delta$. The difficulty arises from the fact that convergence towards $\partial\De$ is also possible. In Section \ref{sect_bord}, we explain how  several sites can a.s.~asymptotically fall out of use (i.e. $x_i^n\to0$) on a network, although all of them performs infinitely many successful communications. The behavior in the neighborhood of $\partial\De$ is difficult to analyse, since the functions $F$ and $H$ are not continuous on $\partial\De$, and since the error in the stochastic approximation can be irregular. 

The paper is organized as follows. In the following Section~\ref{sect_ode}, we explicit the ODE associated to the dynamics of $x_n$, and prove that $H$ is a Lyapunov function for the deterministic dynamics. Section~\ref{sect_p1} and Section~\ref{sect_p2} are respectively dedicated to the proofs of Theorem~\ref{th_pot} and Theorem~\ref{th_eq}. Section~\ref{sect_bord} yields some explanation on the technical difficulties arising from the possible convergence to $\partial\Delta$. Section~\ref{sect_eq} is devoted to the classification of equilibria and analysis of stability, and Section~\ref{sect_p3} concerns the proof of Theorem~\ref{th_posprob}.


\section{Stochastic Approximation}


Let us first compute the conditional increment of the process $(x_n)$ and derive some mean-field ODE. Using stochastic approximation techniques, we will then study this ODE (in particular its equilibria in Section~\ref{sect_eq}), following the idea that the random process approximates the solutions of the ODE.\\

In order to compute the increment of $(x_n)_{n\in\mathbb{N}}$ at
time $n$ we define, for all $i,j\in\V$, $i\sim j$, and for all $n\in\mathbb{N}$, the variables
\beq
\Delta_{ij}^{n+1}&:=& V_{ij}^{n+1}-V_{ij}^n,\label{def_del.ij}\\
\Delta_T^{n+1}&:=&T_{n+1}-T_n.\label{def_del.T}
\eeq

Now, for all $ij\in E$, we have
\begin{align}
\nonumber&x_{ij}^{n+1}-x_{ij}^n = \frac{V_{ij}^{n}+\Delta_{ij}^{n+1}}{T_n+\Delta_T^{n+1}}-\frac{V_{ij}^{n}}{T^{n}}=\frac{T_n\cdot\Delta_{ij}^{n+1}-V_{ij}^n\cdot\Delta_T^{n+1}}{T_n\left(T_n+\Delta_T^{n+1}\right)}\\
&=\frac{\Delta_{ij}^{n+1}-x_{ij}^n\cdot\Delta_T^{n+1}}{T_{n+1}}
\label{firstap}
=\frac{1}{T_n}\left(\Delta_{ij}^{n+1}-x_{ij}^n\cdot\Delta_T^{n+1}\right)+\widetilde{R}_{n+1}^{ij},
\end{align}
where we define $(\widetilde{R}_{n+1})=(\widetilde{R}_{n+1}^{ij})_{i,j\in\V}$ by
\beq\label{def_R-net}
\widetilde{R}_{n+1}^{ij}:=\left(\Delta_{ij}^{n+1}-x_{ij}^n\cdot\Delta_T^{n+1}\right)\left(\frac{1}{T_{n+1}}-\frac{1}{T_{n}}\right).
\eeq
$\widetilde{R}_{n}^{ij}$ is the increment of a bounded and almost surely converging process: indeed, $0\le\Delta_{ij}^{n+1}\le a_{ij}$,  $0\le\Delta_{T}^{n+1}\le \sum_{k,  l\in\V} a_{kl}$, so that, for all $k\in\N$,  
\[
\label{ubr}
\sum_{n\ge k,\, i, j\in\V} \left|\widetilde{R}_{n+1}^{ij}\right| \le 2\sum_{n\ge k,\, i, j\in\V}a_{kl}\left(\frac{1}{T_{n}}-\frac{1}{T_{n+1}}\right)
\le\frac{2\sum_{k,l\in\V}a_{kl}}{T_k}.
\]

Taking the conditional expectation of \eqref{firstap} and using \eqref{proba_succes} and \eqref{incr_T}, we have
\[
\label{stoch_app}
\E(x_{ij}^{n+1}-x_{ij}^n|\mathcal{F}_n)=\frac{x_{ij}^n}{T_n}\big(a_{ij}p_{ij}\frac{x^n_{ij}}{x_i^nx_j^n} - H(x_n)\big) + R_n^{ij}
=\frac{F(x_n)_{i,j}}{T_n}+ R_n^{ij},
\]
where $R_n^{ij}:=\E\left(\left.\widetilde{R}_{n+1}^{ij}\right|\mathcal{F}_n\right)$, 
and  $F$ is defined in \eqref{deff}.
Note that $\sum_n |R_n^{ij}|<\infty$ a.s.,  using \eqref{ubr} and a generalised version of Conditional Borel-Cantelli Lemma, see \cite[Lemma 2.7.33]{DCD2}.

The following Lemma \ref{lem_StochApp} can be deduced from the estimates above; we use the $L^1$ Euclidean norm $|\cdot|$ on $\mathbb{R}^{|\V|\times|\V|}$.

\begin{lem}\label{lem_StochApp}
 There exists an adapted martingale increment process $(\eta_n)_{n\in \mathbb{N}}$ such that, for all $n\in N$,
\begin{align}\label{Stochapprox}x_{n+1}-x_n=\frac{F(x_n)}{T_n}+\eta_{n+1}+\widetilde{R}_{n+1},\end{align}
and, for all $k\in\N$, 
\begin{equation}
\label{uber}
|\eta_{n+1}|\le2\frac{\sum_{i,j\in\V}a_{ij}}{T_n},\,\,
\sum_{n\ge k}|\widetilde{R}_{n+1}|\le\frac{2\sum_{k,l\in\V}a_{kl}}{T_k}.
\end{equation}
\end{lem}

\begin{proof}
Note that 
\begin{equation}
\label{eup}
|x_{n+1}-x_n-\widetilde{R}_{n+1}|=\left|\frac{1}{T_n}\left(\Delta_{ij}^{n+1}-x_{ij}^n\cdot\Delta_T^{n+1}\right)\right|
\le\frac{\sum_{i,j\in\V}a_{ij}}{T_n},
\end{equation}
which provides the upper bound on $|\eta_{n+1}|$; the upper bound on $(R_n)_{n\in\N}$ is given by \eqref{ubr}. 
\end{proof}

Corollary~\ref{cor_Tn}, which proves asymptotic linear growth of $T_n$, will in particular imply that the martingale
$(\sum_{k=1}^n\eta_{k})_{n\in\mathbb{N}}$ converges a.s.~by Doob's convergence theorem.

Equation ($\ref{Stochapprox}$) is a stochastic approximation of $(x_n)$. We will show in Corollary~\ref{cor_Tn} that $1/T_n$ is the step size and is of the order of $1/n$. Therefore it is reasonable to expect that $(x_n)$ converges to the set of equilibria of \eqref{ode}, which we show in Theorem \ref{th_eq}. 

We first study the evolution of the deterministic process driven by~\eqref{ode}. Recall the definition \eqref{def_Gamma} of $\Gamma$, the \emph{set of
equilibria}
 of the ODE~\eqref{ode}.


\subsection{Analysis of the mean-field ODE}\label{sect_ode}


In this section we show that the function $H$ is a Lyapunov function for the ODE~\eqref{ode}, i.e. that $H$ is nondecreasing along the paths of the ODE. It is natural to expect that result, in the sense that the overall expected payoff of the network should indeed increase in average; see \cite{HST} for more details.

We will then deduce that the random process $H(x_n)$ is a submartingale up to an error term (see Theorem~\ref{th_pot}). Recall that this stochastic result does not directly follow from the deterministic statements and classical results, since $H$ is not continuous on the boundary $\partial\Delta$.\\

Let, for all $x\in\De$,
\beq\label{def_p}
p(x):=\sum_{i,j,k\in \V:x_{ij},x_{ik}>0}\frac{x_{ij}x_{ik}}{x_i}\left(y_{ij}-y_{ik}\right)^2,
\eeq
where, given  $x\in\Delta\setminus\partial\Delta$ we let, for all $i,j\in\V$ such that $i\sim j$,
\beq\label{def_y}
y_{ij}:=a_{ij}p_{ij}\frac{x_{ij}}{x_ix_j}
\eeq
be the weighted \emph{efficiency}  of the pair $ij$.

\begin{prop}\label{prop_Lya}
$H$ is a Lyapunov function on
$\Delta\setminus\partial\Delta$ for the mean-field ODE~\eqref{ode}, more precisely,
\begin{equation}
\label{eq:varh1}
\nabla H \cdot F \,(x)\, =\, p(x)\ge0,
\end{equation}
where $\nabla H=(\partial H/\partial x_{ij})_{i,j\in\V}$.
\end{prop}
\noindent\textbf{Remark.} \textnormal{ $H$ is not a strict
Lyapunov function, i.e. $\nabla H\cdot F$ can vanish outside $\Gamma$. 
}

\begin{proof}
Note that, in the definition of $\nabla H$, $x_{ij}$ and $x_{ji}$ are independent variables. For any $x\in\Delta\setminus\partial\Delta$, we have

\beqq
\nabla H\cdot F(x)&=&  \sum_{i,j\in\V} a_{ij}p_{ij}\left\{ 2 \frac{x_{ij}^2}{x_ix_j}\left(a_{ij}p_{ij}\frac{x_{ij}}{x_ix_j}-H(x)\right)\right.\\
&&-\frac{x_{ij}^2}{x_i^2x_j}\left(\sum_{k\sim i}a_{ik} p_{ik}\frac{x_{ik}^2}{x_ix_k}-x_{ik}H(x)\right)\\
&&\left. -\frac{x_{ij}^2}{x_ix_j^2}\left(\sum_{l\sim j} a_{jl}p_{jl} \frac{x_{jl}^2}{x_jx_l}-x_{jl}H(x)\right)\right\}\\
&=& 2\sum_{i,j\in\V} a_{ij}^2p_{ij}^2 \frac{x_{ij}^3}{(x_ix_j)^2}\\
&&-\sum_{i,j,k\in\V} a_{ij}p_{ij}\frac{x_{ij}^2}{x_i^2x_j}a_{ik} p_{ik}\frac{x_{ik}^2}{x_ix_k}-\sum_{i,j,l\in\V} a_{ij}p_{ij}\frac{x_{ij}^2}{x_ix_j^2}a_{jl} p_{jl}\frac{x_{jl}^2}{x_jx_l}\\
&&-H(x)\times\sum_{i,j\in\V} a_{ij}p_{ij}\frac{x_{ij}^2}{x_ix_j}\left(2-\frac{1}{x_i}\sum_{k\sim i}x_{ik}-\frac{1}{x_j}\sum_{l\sim j}x_{jl}\right)\\
&=& 2\sum_{i,j\in\V} a_{ij}^2p_{ij}^2 \frac{x_{ij}^3}{(x_ix_j)^2} - 2\sum_{i,j,k\in\V} a_{ij}p_{ij}a_{ik}p_{ik}\frac{x_{ij}^2x_{ik}^2}{x_i^3x_jx_k}\\
&=& \sum_{i,j,k\in\V} a_{ij}^2p_{ij}^2 \frac{x_{ij}^3x_{ik}}{x_i^3x_j^2}+ \sum_{i,j,k\in\V} a_{ik}^2p_{ik}^2 \frac{x_{ik}^3x_{ij}}{x_i^3x_k^2}\\
&&-  \sum_{i,j,k\in\V} 2a_{ij}p_{ij}a_{ik}p_{ik}\frac{x_{ij}^2x_{ik}^2}{x_i^3x_jx_k}\\
&=&  \sum_{i,j,k\in\V}\frac{x_{ij}x_{ik}}{x_i}\bigg[a_{ij}p_{ij}\frac{x_{ij}}{x_ix_j}-a_{ik}p_{ik}\frac{x_{ik}}{x_ix_k}\bigg]^2\\
&=&p(x)\ge 0.
\eeqq

We used, in the third equality, that $\sum_{k\sim i}x_{ik}=x_i$ (and the same equality with $j$ instead of $i$), and we also used relevant substitutions and permutations of the indices. In the last equality, we use definition~\eqref{def_y} of $y_{ij}$.

\end{proof}

Given  $x\in\Delta\setminus\partial\Delta$, and $i,j\in\V$ such that $i\sim j$, let
\beq\label{def_N}
N_i(x):=\sum_{k\in\V}\frac{x_{ik}}{x_i}\cdot y_{ik}
\eeq
be the \emph{weighted efficiency} $N_i(x)$ of $i$. Recall that $y_{ik}$ is defined in \eqref{def_y}.

\begin{lem} \label{rec-dl}
For any $x\in\Delta\setminus\partial\Delta$:
\beq
p(x)=\nabla H\cdot F(x)=2\sum_{i,j\in\V} x_{ij}\Big(y_{ij}-N_i(x)\Big)^2.\label{Jz1}
\eeq
\end{lem}

\noindent\textbf{Remark.}\textnormal{
In the context of communication
systems, the above two formulas \eqref{eq:varh1}, \eqref{Jz1} mean that  the growth rate
of the expected payoff is a function depending on the difference
between efficiencies of different strategy pairs.}

\begin{proof}[Proof of Lemma~\ref{rec-dl}]
Fix $i\in\V$, and define a  measure $\mathbb{P}_i$ and a random variable $Y$ such that, for all $k\in\V$, $\mathbb{P}_i(Y=y_{ik}):=\frac{x_{ik}}{x_i}$. We denote $\E_i$ the expectation associated with $\mathbb{P}_i$.
Recalling~\eqref{def_N}, we have
$$\E_i(Y)= N_i(x).$$
This implies that, for all $j\sim i$,
\beqq
\E_i\Big[\big(y_{ij}-Y\big)^2\Big]&=& \Big(y_{ij}-N_i(x)\Big)^2+ \E_i\Big[\big(Y-N_i(x)\big)^2\Big]\\
&=& \Big( y_{ij}-N_i(x)\Big)^2+ \sum_k \frac{x_{ik}}{x_i} \Big(N_i(x)-y_{ik}\Big)^2.
\eeqq
On the other hand,
\beqq
\E_i\Big[\big(y_{ij}-Y\big)^2\Big]&=&\sum_{k}\frac{x_{ik}}{x_i}\left(y_{ij}-y_{ik}\right)^2.
\eeqq

Therefore, for any $x\in\Delta\setminus\partial\Delta$,
\beqq
\nabla H\cdot F(x)&=& \sum_{i,j,k\in\V}\frac{x_{ij}x_{ik}}{x_i}\bigg[y_{ij}-y_{ik}\bigg]^2
=\sum_{i,j\in\V}x_{ij}\sum_k\frac{x_{ik}}{x_i}\bigg[y_{ij}-y_{ik}\bigg]^2\\
&=&\sum_{i,j\in\V} x_{ij} \Big( y_{ij}-N_i(x)\Big)^2+ \sum_{i,j,k\in\V} \frac{x_{ij}x_{ik}}{x_i} \Big(N_i(x)-y_{ik}\Big)^2\\
&=&2 \sum_{i,j\in\V} x_{ij} \Big( y_{ij}-N_i(x)\Big)^2,
\eeqq
switching labels of $k$ and $j$ on the last sum of the penultimate line.
\end{proof}

%
%

Let us
define
\beq\label{def_Lam}
\Lambda:=\{x\in\Delta:\,p(x)=0\},
\eeq
where $p$, defined in~\eqref{def_p}, is the derivative of $H$ along a trajectory of the ODE~\eqref{ode}. The proof of the following Lemma \ref{restpoint} is straightforward.
\begin{lem}\label{restpoint}
$x \in \Lambda$ if and only if 
$$y_{ij}=y_{ik}, \textrm{ for all } i,j,k\in\V \text{ s.t. } x_{ij}\ne 0, x_{ik}
\ne 0.
$$
\end{lem}
\begin{rem}\label{rem1} 
$x\in\Lambda$ iff the weighted efficiencies on edges $y_e$  are constant over the connected components of $(G_x,\simx)$. Now $x\in\G$ iff $y_e$ is equal to $H(x)$ for any edge $e$ of $G_x$. Therefore $\Gamma\subseteq\Lambda$ but $\Lambda\neq\Gamma$ in general, and $H$ is not a strict Lyapounov function. 
\end{rem}


\subsection{Proof of Theorem~\ref{th_pot} and asymptotic linear growth of $(T_n)$}\label{sect_p1}


\begin{proof}[Proof of Theorem~\ref{th_pot}]
In order to compute the conditional expectation of the increment of $(H(x_n))$, let us first estimate the increment
$$\xi_{ij}^{n+1}=\frac{(V_{ij}^{n+1})^2}{V_i^{n+1}V_j^{n+1}}-\frac{(V_{ij}^{n})^2}{V_i^{n}V_j^{n}}.$$
By Taylor expansion, 
$$\frac{(1+a)^2}{(1+b)(1+c)}=1+2a-b-c+O(a^2+b^2+c^2)$$
for all $a$, $b$, $c$ $>0$.

We use notation $\De_{ij}^{n+1}$ from \eqref{def_del.ij}, and also define $\Delta_{i}^{n+1}:=\sum_j \Delta_{ij}^{n+1}$. 
We have
\begin{align*}
\xi_{ij}^{n+1}&=\frac{(V_{ij}^{n})^2}{V_i^{n}V_j^{n}}\left[
\frac{(1+\De_{ij}^{n+1}/V_{ij}^n)^2}{(1+\De_{i}^{n+1}/V_i^n)(1+\De_{j}^{n+1}/V_j^n)}-1\right]\\
&=\frac{(V_{ij}^{n})^2}{V_i^{n}V_j^{n}}\left(2\frac{\De_{ij}^{n+1}}{V_{ij}^n}-\frac{\De_{i}^{n+1}}{V_i^n}
-\frac{\De_{j}^{n+1}}{V_j^n}+\e_{n+1}^{ij}\right)\\
&=\frac{2V_{ij}^n\De_{ij}^{n+1}}{V_i^nV_j^n}-\frac{(V_{ij}^n)^2\De_i^{n+1}}{(V_i^n)^2V_j^n}-\frac{(V_{ij}^n)^2\De_j^{n+1}}{V_i^n(V_j^n)^2}+\zeta_{n+1}^{ij},
\end{align*}
where $|\e_{n+1}^{ij}|$ and  $|\zeta_{n+1}^{ij}|$ are upper bounded by
\begin{equation}
C\left(\left(\frac{\De_{ij}^{n+1}}{V_{ij}^{n}}\right)^2+\left(\frac{\De_{i}^{n+1}}{V_{i}^{n}}\right)^2
+\left(\frac{\De_{j}^{n+1}}{V_{j}^{n}}\right)^2\right),
\end{equation}
for some universal constant $C>0$.

Note that, since the harmonic series converges we have, for all $i$, $j$ $\in\V$,
$$\sum_{n\in\N}|\zeta_{n+1}^{ij}|<\infty.$$

Let $\zeta_{n+1}=\sum_{i,j\in\V}a_{ij}p_{ij}\zeta_{n+1}^{ij}$.

On the other hand, summing our estimate over $i$, $j$, we deduce that
\beqq
H(x_{n+1})&-&H(x_n)-\zeta_{n+1}\\
&=& 2\sum_{i,j\in\V}a_{ij}p_{ij}\frac{V_{ij}^n\Delta_{ij}^{n+1} }{V_i^nV_j^n}       -2\sum_{i,j\in\V}a_{ij}p_{ij}\frac{(V_{ij}^n)^2\Delta_{i}^{n+1}}{ (V_i^n)^2V_j^n    }\\
&=&\sum_{i,j,k\in\V}a_{ij}p_{ij}\frac{V_{ij}^n V_{ik}^n}{(V_i^n)^2V_j^n}  \Delta_{ij}^{n+1}    +\sum_{i,j,k\in\V}a_{ik}p_{ik}\frac{V_{ik}^n V_{ij}^n}{(V_i^n)^2V_k^n}  \Delta_{ik}^{n+1}    \\
&& -2\sum_{i,j,k\in\V}a_{ij}p_{ij}\frac{(V_{ij}^n)^2}{ (V_i^n)^2V_j^n    }\Delta_{ik}^{n+1}.
\eeqq
Recall that $\E\left(\Delta_{ik}^{n+1}|\mathcal{F}_n\right)=a_{ik}p_{ik}\frac{(x_{ik}^n)^2}{x_i^nx_k^n}$, for all $i,k\in\V$.
Finally, taking the conditional expectation yields
\begin{align}\nonumber
&\E\Big[H(x_{n+1})-H(x_n)-\zeta_{n+1}\big|\mathcal{F}_n\Big]\\ \nonumber
&=\sum_{i,j,k\in\V} \frac{V_{ij}^n V_{ik}^n}{V_i^n}\left[  \left(a_{ij}p_{ij}\frac{V_{ij}^n}{V_i^nV_j^n}\right)^2 +\left(a_{ik}p_{ik}\frac{V_{ik}^n}{V_i^nV_k^n}\right)^2-2 a_{ij}p_{ij} a_{ik}p_{ik}\frac{V_{ij}^nV_{ik}^n}{\left(V_i^nV_j^n\right)^2}\right]\\ \label{incr-error}
&=\frac{1}{T_n}\sum_{i,j,k\in\V} \frac{x_{ij}^n x_{ik}^n}{x_i^n}\left[ a_{ij}p_{ij}\frac{x_{ij}^n}{x_i^nx_j^n}-a_{ik}p_{ik}\frac{x_{ik}^n}{x_i^nx_k^n}\right]^2
=\frac{1}{T_n}p(x_n),
\end{align}
where $p$ is defined in \eqref{def_p}. We conclude by noting that $(H(x_n))$ is the sum of a converging process and a bounded submartingale.
\end{proof}

We have proved that the expected payoff process converges almost surely. We can now prove that the process $(x_n)$ converges to $\Lambda$, defined in~\eqref{def_Lam}.
\begin{prop}\label{prop_p0}
$(x_n)_{n\in\mathbb{N}}$ converges a.s.~to
$\Lambda$. More precisely, $(p(x_n))_{n\in\mathbb{N}}$ converges to $0$ a.s.
\end{prop}
Define for all $\varepsilon>0$ the set
\beq\label{def_Delteps}
\Delta_{\varepsilon}:=\{x\in\Delta\setminus\partial\Delta:\,
p(x)>\varepsilon\}.
\eeq

\begin{proof}This argument is similar to the proof of convergence to the set of equilibria in \cite{APSV}. Let, for all $n\in\N$, 
$$\xi_n=\sum_{k=0}^n\zeta_k.$$

We have proved that $(H(x_n)-\xi_n)$ is a bounded submartingale. Thanks to the Doob decomposition, we write it as the sum of a martingale $(M_n)$ and an increasing predictable process $(A_n)$. For all $n\in\mathbb{N}$, we have:
$$H(x_n)-\xi_n=M_n+A_n.$$
$(M_n)$ is an upper-bounded martingale, hence a.s.~converges. On the other hand, using \eqref{incr-error},
$$
A_{n+1}-A_n= \E\Big[H(x_{n+1})-H(x_n)-\zeta_{n+1}\big|\mathcal{F}_n\Big]=\frac{1}{T_n}p(x_n).
$$

Fix $\varepsilon>0$ and define $\delta>0$ the distance between the sets $\Delta_\varepsilon$ and $\Delta_{\varepsilon/2}^c$. First, notice that $x_n$ cannot stay in $\Delta_\varepsilon$ for ever as $1/T_n$ is not summable.\\
Then, assume that $x_n\in\Delta_\varepsilon$, $x_{n+1},...,x_{n+k-1}\in\Delta_{\varepsilon/2}\cap\Delta_\varepsilon^c, x_{n+k}\in\Delta_{\varepsilon/2}^c$, then:
$$A_{n+k}-A_n= \sum_{r=n}^{n+k-1} \frac{p(x_r)}{T_r} \ge \varepsilon \sum_{r=n}^{n+k-1} \frac{1}{2T_r}.$$
Therefore, using \eqref{eup} and Lemma \ref{lem_StochApp}, we have
$$\delta\le \sum_{r=n}^{n+k-1}|x_{r+1}-x_r|\le \sum_{r=n}^{n+k-1}\frac{\Cst(a)}{T_r}\le \frac{\Cst(a)}{\varepsilon} (A_{n+k}-A_n).$$
Then, if $(x_n)$ were infinitely often away from $\Lambda$, $A_n$, and consequently $H(x_n)$, would explode. This implies the conclusion.
\end{proof}

We end this section by stating a result that implies that the process $(T_n)$ has a linear asymptotic growth.

Define the constant
\beq\label{singe}
w_{\min}=\min\left\{a_{ij}p_{ij}:a_{ij}p_{ij}>0\right\}.
\eeq
\begin{coro} \label{cor_Tn} There exists a constant $h_{\min}=\Cst((V_{ij}^0)_{i,j\in\V})>0$ such that, a.s.
  $$\frac{T_n}{n} \rightarrow \lim_{n\to\infty}H(x_n) \,\in\,\left[h_{\min}w_{\min}, \sum_{i,j\in\V}a_{ij}p_{ij}\right] \text{ as } n\to \infty.$$
Moreover, for all $i\in\V$ such that $a_{ij}p_{ij}\ne0$ for some $j\sim i$, $n\in\N$, 
\begin{equation}
\label{lbni}
N_i(x_n)=\sum_{j\in \V, \,x_j>0}a_{ij}p_{ij} \frac{(x_{ij}^n)^2}{(x_i^n)^2x_j^n}\ge h_{\min}w_{\min}.
\end{equation}
\end{coro}

\begin{proof}
Given $x\in\De$, recall definition~\eqref{def_N} of $N_i(x)$. Using that $\sum_{j\in \V} x_j=1$, we have, by Cauchy-Schwarz inequality, if $x_i\ne0$, 
\begin{align*}
N_i(x)&=\left(\sum_{j\in \V, \,x_j>0} a_{ij}p_{ij}\frac{x_{ij}^2}{x_i^2x_j}\right)\sum_{j\in \V} x_j\ge w_{\min} \Big(\sum_{j: a_{ij}p_{ij}>0, \,x_j>0} \frac{x_{ij}}{x_i}\Big)^2\\
\end{align*}
Note that, for all $j\sim i$ with $a_{ij}p_{ij}=0$, we have $V^n_{ij}=v^0_{ij}$ a.s. for all $n\in\mathbb{N}$. Hence, almost surely
\begin{align*}
\sum_{j: a_{ij}p_{ij}>0} \frac{x^n_{ij}}{x^n_i}=1-\sum_{j: a_{ij}p_{ij}=0} \frac{v^0_{ij}}{V^n_i}
\ge1-\sum_{j: a_{ij}p_{ij}=0} \frac{v^0_{ij}}{V^0_i}=\sum_{j: a_{ij}p_{ij}>0} \frac{x^0_{ij}}{x^0_i}.
\end{align*}
Moreover, the last sum is lower bounded by  some constant $h_0>0$ as soon as there exists some $j\sim i$ such that $a_{ij}p_{ij}>0$. This provides the lower-bound on $N_i(x)$.\\
Similarly, for any $x\in\Delta$, defined in \eqref{def_delta}, we have, by Cauchy-Schwarz inequality and using that $\sum_{i,j} x_ix_j=1$,
\begin{align}\label{lbHx}
H(x)\ge w_{\min}\Big(\sum_{i,j:a_{ij}p_{ij}>0}x_{ij}\Big)^2\ge w_{\min}h_1^2,
\end{align}
where $h_1>0$ is defined in \eqref{def_h1}. This gives us a lower-bound on $\liminf_n H(x_n)$. The upper bound on $H$ is trivial.\\

Finally, recall that 
\[
\E\left(\left.T_{n+1}-T_n\right|\mathcal{F}_n\right)=H(x_n)
\]
and the result is now a direct consequence of  Theorem~\ref{th_pot} and a generalised version of Conditional Borel-Cantelli Lemma, see \cite[Lemma 2.7.33]{DCD2}.
\end{proof}


\subsection{Proof of Theorem~\ref{th_eq}}\label{sect_p2}


In the previous section, we proved that the occupation measure $(x_n)$ converges a.s.~to $\Lambda$, which is the set of points where the derivative of the Lyapunov function $H$ vanishes. As we already mentioned in Remark~\ref{rem1}, $H$ is not a strict Lyapunov function, i.e. $\Gamma\subseteq \Lambda$ but $\Gamma\neq \Lambda$, where $\Gamma$ is the set of equilibria of the ODE~\eqref{ode}, defined in~\eqref{def_Gamma}.

The aim of this section is to prove that the occupation measure converges a.s.~to the set of equilibria of the ODE~\eqref{ode}. To that end we show the following Proposition \ref{prop:setofequ} in Section \ref{s:setofequ}.\\

\begin{prop} \label{prop:setofequ}
Let $i$, $j$ $\in\V$, $i \sim j$.
Then $$\limsup_{n\rightarrow\infty} x_{ij}^n\big(y_{ij}^n-H(x_n)\big)^{-}=0.$$
\end{prop}

Proposition \ref{prop:setofequ} will enable us to conclude. Indeed, for all $n\in\mathbb{N}$, we have
\begin{align*}
H(x_n)=\sum_{i,j\in\V}y_{ij}^nx_{ij}^n= H(x_n)+\sum_{i,j\in\V }(y_{ij}^n-H(x_n))^+x_{ij}^n-\sum_{i,j\in\V}(y_{ij}^n-H(x_n))^-x_{ij}^n
\end{align*}
Now,  for all $i$, $j$ $\in\V$, Proposition
\ref{prop:setofequ} implies that $\lim_{n\to \infty} x_{ij}^n(y_{ij}^n-H(x_n))^{-}=0$. Subsequently, for all $i$, $j$ $\in\V$,
\begin{equation}
\label{xyr}
\lim_{n\to \infty} x_{ij}^n(y_{ij}^n-H(x_n))=0,
\end{equation}
so that $F(x_n)$ converges to $0$ as $n\to\infty$.

Let us show that this implies convergence of $(x_n)_{n\in\N}$ to the set of equilibria $\G$. Recall that $|\cdot|$ denotes the $L^1$-norm on $\mathbb{R}^{\V\times\V}$.
Let $L$ be the limit-set of $(x_n)_{n\in\N}$, which is a non-empty subset of $\Delta$ by compactness. Fix $\e>0$. Then there exists $n_0$ such that $|x_n-L|<\e$ for any $n\ge n_0$ (otherwise there would be a limit-point at distance at least $\e$ from $L$, by compactness). For any $n\ge n_0$, there exists $z\in L$ (that may depend on $n$) such that $|z-x_{n}|<\e$: let us prove that $z\in \Gamma$.\\
By contradiction, assume that there exist $i_0,j_0$ such that $z_{i_0j_0}>0$ and $|a_{i_0j_0}p_{i_0j_0}z_{i_0j_0}/z_{i_0}z_{j_0}-H(z)|=c_z>0$.
Fix $\delta\in (0,\min\{z_{ij}/4:z_{ij}>0\})$. Let $N\in \mathbb{N}$ be such that  $|x_{ij}^k\big(y_{ij}^k-H(x_k)\big)|<\delta$ for any $k\ge N$.
As $z\in L$, there exists $k\ge N$ such that $|x_k-z|<\delta$ which implies, using that $\delta<\min\{z_{ij}/4:z_{ij}>0\}$,  that $x^k_{ij}>\delta$ if and only if $z_{ij}>0$, for any $i,j\in\V$. In this case, it is straightforward to check that
\[
a_{ij}p_{ij}\left|\frac{x_{ij}^k}{x_i^kx_j^k}-\frac{z_{ij}^k}{z_i^kz_j^k}\right|,a_{ij}p_{ij}\left|\frac{(x_{ij}^k)^2}{x_i^kx_j^k}-\frac{(z_{ij}^k)^2}{z_i^kz_j^k}\right|\le \delta \verb?Cst?((a_{ml},p_{ml})_{m,l\in\V},z).
\]
Besides, if $x^k_{ij}\le\delta$ then $a_{ij}p_{ij}(x_{ij}^k)^2/x_i^kx_j^k\le \delta\verb?Cst?((a_{ml},p_{ml})_{m,l\in\V})$, using that $H$ is bounded and that $|x_{ij}^k\big(y_{ij}^k-H(x_k)\big)|<\delta$.\\
We deduce subsequently that $|H(x_k)-H(z)|<\delta \verb?Cst?((a_{ml},p_{ml})_{m,l\in\V},z)$. Finally, as soon as $\delta$ is small enough, we obtain a contradiction by noting that
\begin{eqnarray*}
\left|a_{i_0j_0}p_{i_0j_0}\frac{z_{i_0j_0}}{z_{i_0}z_{j_0}}-H(z)\right|&\le& \left|a_{i_0j_0}p_{i_0j_0}\frac{x_{i_0j_0}^k}{x_{i_0}^kx_{j_0}^k}-H(x_k)\right|+\delta \verb?Cst?((a_{ml},p_{ml})_{m,l\in\V},z)\\
&\le& \delta \verb?Cst?((a_{ml},p_{ml})_{m,l\in\V},z)<c_z.
\end{eqnarray*}
Therefore, $z\in\Gamma$.\\

\medskip

Let us now sketch the proof of Proposition \ref{prop:setofequ}, given in Section \ref{s:setofequ}. Define
\begin{equation}
U_{ij}(\eps):= \big\{x\in\Delta: x_{ij}<\eps \text{ or } y_{ij}-H(x)\ge -\eps\big\},\label{def_Uij}\\
\end{equation}
and  recall the definitions of $\Delta_{\varepsilon}$ and $h_{\min}$ respectively introduced in \eqref{def_Delteps}  and in Corollary \ref{cor_Tn}.

Proposition \ref{prop:setofequ} will follow from Lemma \ref{lem53}, which implies that both $x_{ij}^n$ (resp. $y_{ij}^n$) decreases (resp. increases) outside $U_{ij}(\eps)\cup\De_{\eps^4}$. 

\begin{lem} \label{lem53} 
Let $i$, $j$ $\in\V$, $i \sim j$,  with $a_{ij}p_{ij}>0$. Let
\beqq
Y_n&:=& \sum_{m=m_0}^n \Big(y_{ij}^m-y_{ij}^{m-1}-a_{ij}p_{ij}\frac{\eps^2}{6T_{m-1}}\Big)\1_{\{x_{m-1}\notin U_{ij}(\eps)\cup\Delta_{\eps^4}\}};\\
S_n&:=& \sum_{m=m_0}^n \Big(x_{ij}^m-x_{ij}^{m-1}+\frac{\eps^2}{2T_{m-1}}\Big)\1_{\{x_{m-1}\notin U_{ij}(\eps)\}}.
\eeqq
Assume $T_{m_0}\ge \Cst(a_{ij}p_{ij},a,\e)$ and $\e\in(0,1/9)$. Then
\begin{enumerate}
\item $(Y_n)_{n\ge m_0}$ (resp. $(S_n)_{n\ge m_0}$) is a submartingale (resp. supermartingale);\\
\item $\limsup_{n\ge m, m\rightarrow \infty}(Y_n-Y_m)^{-}=\limsup_{n\ge m, m\rightarrow \infty}(S_n-S_m)^{+}=0$.
\end{enumerate}
\end{lem}

Lemma \ref{lem53} is shown is Section \ref{p:lem53}. It follows from Lemma \ref{lem_StochApp} for its part on $x_{ij}^n$, and from the following two Lemmas \ref{stochapp2} and \ref{lem51} for its part on $y_{ij}^n$. 

Lemma \ref{stochapp2} is proved in Section \ref{p:stochapp2}, and shows that the increment of $y_{ij}^n$ is driven by a differential equation of the type $\dot{y}=G(x)$, where, for all $x\in\De$, we let 
$$G(x)=\left(y_{ij}\Big(y_{ij}-N_i(x)-N_j(x)+H(x)\Big)\right)_{i,j\in\V:\, x_{ij}>0}.$$
Outside $U_{ij}(\eps)\cup\De_{\eps^4}$,  Lemma \ref{lem51}  implies that $$G(x)_{ij}\approx y_{ij}(H(x)-y_{ij}),$$
which indeed yields that $y_{ij}^n$ increases on average.
\begin{lem} \label{stochapp2}
For all $i$, $j$ $\in\V$, the increment of $(y_{ij}^n)_n$ expands as follows:
\begin{equation}
\label{defy}
y_{ij}^{n+1}-y_{ij}^n= \frac{1}{T_n} G(x_n)_{ij} + r_{ij}^{n+1} + \zeta_{ij}^{n+1},
\end{equation}
where $(r_{ij}^n)$ is predictable, $\E(\zeta_{ij}^{n+1}|\mathcal{F}_n)=0$ and 
$$|y_{ij}^{n+1}-y_{ij}^{n}|,\,\,|\zeta_{ij}^{n+1}|\le \frac{c_0}{T_nx_i^nx_j^n}\text{, and } |r_{ij}^{n+1}|\le \frac{c_0}{(T_nx_i^nx_j^n)^2},$$
where $c_0=\Cst(a)$.
\end{lem}
\begin{lem} \label{lem51}
Assume that $\eps\in(0,1/9)$, and that $x\in\Delta_{\eps^4}^c$. Then, for all vertices $i$ and $j$, $i \sim j$, such that $x_{ij}>\eps$, we have:
$$|y_{ij}-N_i(x)|<\frac{\eps}{3} \text{ and } |y_{ij}-N_j(x)|<\frac{\eps}{3}.$$
\end{lem}

\begin{proof}
Follows directly from Lemma \ref{rec-dl}.
\end{proof}

\subsubsection{Proof of Lemma \ref{stochapp2}}
\label{p:stochapp2}
Using $(1+c)^{-1}=1-c+O(c^2)$ for all $c\ge 0$, we deduce that
\begin{equation}
\label{dl4}
f(s,t,u,v)=\frac{(1+s)(1+t)}{(1+u)(1+v)}=1+s+t-u-v+O(\eta(\e+\eta))
\end{equation}
if $t\in[0,\e]$  and $s$, $u$, $v$ $\in[0,\eta]$.

Now recall $y_{ij}=a_{ij}p_{ij}T_nV_{ij}^n/(V_i^nV_j^n)$: therefore,
$$y_{ij}^{n+1}=y_{ij}^nf(s_{n+1},t_{n+1},u_{n+1},v_{n+1}),$$
where
\begin{align*}
&s_{n+1}=\frac{T_{n+1}-T_n}{T_n},\,\,t_{n+1}=\frac{V_{ij}^{n+1}-V_{ij}^n}{V_{ij}^n},\\
&u_{n+1}=\frac{V_{i}^{n+1}-V_{i}^n}{V_{i}^n},\,\,v_{n+1}=\frac{V_{j}^{n+1}-V_{j}^n}{V_{j}^n}.
\end{align*}

Let 
$$\de_{ij}^{n+1}=y_{ij}^{n}(s_{n+1}+t_{n+1}-u_{n+1}-v_{n+1}),\,\,
\zeta_{ij}^{n+1}=\de_{ij}^{n+1}-\Es(\de_{ij}^{n+1}|\Fn).$$
Note that $\zeta_{ij}^{n+1}$ uniquely determines $r_{ij}^{n+1}$ from \eqref{defy}. 

By definition of $H$, $N_i$ and $y_{ij}$, we have
\begin{align*}
&\Es(s_{n+1}|\Fn)=\frac{H(x_n)}{T_n},\,\,\Es(t_{n+1}|\Fn)=\frac{y_{ij}^n}{T_n},\\
&\Es(u_{n+1}|\Fn)=\frac{N_i(x_n)}{T_n},\,\,\Es(v_{n+1}|\Fn)=\frac{N_j(x_n)}{T_n},
\end{align*}
so that 
$$\Es(\de_{ij}^{n+1}|\Fn)
=\frac{1}{T_n} y_{ij}^n\Big(y_{ij}^n-N_i(x_n)-N_j(x_n)+H(x_n)\Big).$$

On the other hand, $s_{n+1}T_n$, $t_{n+1}V_{ij}^{n}$, $u_{n+1}V_{i}^{n}$ and $v_{n+1}V_{j}^{n}$ are all upper bounded by $\Cst(a)$, which implies that
$$|\de_{ij}^{n+1}|\le\Cst(a)\frac{y_{ij}^n}{V_{ij}^n}=\frac{\Cst(a)}{T_nx_i^nx_j^n}.$$
Also, assume w.l.o.g. that $V_i^n\le V_j^n$; then it follows from \eqref{dl4} that
$$|r_{ij}^{n+1}|\le\Cst (s_{n+1}\vee u_{n+1}\vee v_{n+1})t_{n+1}y_{ij}^n
\le\frac{\Cst(a)}{V_{ij}^nV_i^n}\frac{T_nV_{ij}^n}{V_i^nV_j^n}
\le\Cst(a)\left(\frac{T_n}{V_i^nV_j^n}\right)^2,
$$
which enables us to conclude, noting that $|y_{ij}^{n+1}-y_{ij}^{n}|\le|\de_{ij}^{n+1}|+|r_{ij}^{n+1}|$.
\subsubsection{Proof of Lemma \ref{lem53}}
\label{p:lem53}
Assume $x_n\not\in U_{ij}(\e)\cup\De_{\e^4}$. 

Using  that $|\tilde{R}_{n+1}^{ij}|\le\Cst(a)/T_n^2$ from \eqref{def_R-net}, Lemma \ref{lem_StochApp} implies 
\beqq
\E(x_{ij}^{n+1}-x_{ij}^n|\mathcal{F}_n)&=&\frac{x_{ij}^n}{T_n}\big(y_{ij}^n - H(x_n)\big) +\Es(\tilde{R}_{n+1}^{ij}|\F_n)\\
&\le & -\frac{\eps^2}{T_n}+\frac{\Cst(a)}{T_n^2}\le -\frac{\eps^2}{2T_n},
\eeqq
if $n\ge m_0$ and $T_{m_0}\ge\Cst(a,\e)$, so that $(S_n)_{n\ge m_0}$ is a supermartingale. 

Also, using Lemma \ref{lem51}, 
$$y_{ij}^n-N_i(x_n)-N_j(x_n)+H(x_n)\ge
\frac{\eps}{3}.$$
Hence, using now Lemma \ref{stochapp2},
\begin{align*}
\E\big(y_{ij}^{n+1}-y_{ij}^n|\mathcal{F}_n\big)&\ge\frac{\e}{3}\frac{y_{ij}^n}{T_n}-\frac{c_0}{(T_nx_i^nx_j^n)^2}=\frac{1}{T_nx_i^nx_j^n}\left(\frac{a_{ij}p_{ij}x_{ij}^n\eps}{3}- \frac{c_0}{T_nx_i^nx_j^n}\right)\\
&\ge \frac{1}{T_nx_i^nx_j^n}\left(\frac{a_{ij}p_{ij}\eps^2}{3}- \frac{c_0}{T_n\eps^2}\right)\ge \frac{\Cst(a_{ij}p_{ij},a,\e)}{T_n}
\end{align*}
if  $n\ge m_0$ and $T_{m_0}\ge\Cst(a_{ij}p_{ij},a,\e)$. This concludes the proof of $(1)$.\\

In order to prove $(2)$, define $\Pi_n$ and $\Xi_n$ as the martingale parts in the Doob decompositions of $Y_n$ and $S_n$:
\beqq
\Pi_n&:=& Y_n-\sum_{m=m_0}^n \E\big[Y_m-Y_{m-1}|\mathcal{F}_n\big],\\
\Xi_n&:=& S_n-\sum_{m=m_0}^n \E\big[S_m-S_{m-1}|\mathcal{F}_n\big].
\eeqq
By Lemma~\ref{stochapp2}, we have for all $n\ge m_0$,
\beqq
\E\Big[\big(\Pi_{n+1}-\Pi_n\big)^2\big|\mathcal{F}_n\Big] &\le & \E\Big[\big(\zeta_{ij}^{n+1}\big)^2\1_{\{x_n\notin U_{ij}(\eps)\cup\Delta_{\eps^4}\}}\big|\mathcal{F}_n\Big]\\
&\le&  \frac{c_0^2}{(T_nx_i^nx_j^n)^2} \1_{\{x_n\notin U_{ij}(\eps)\cup\Delta_{\eps^4}\}}\le \frac{c_0^2}{\eps^4T_n^2}.
\eeqq
Therefore, using Corollary \ref{cor_Tn}, $(\Pi_n)_n$ is bounded in $\mathbb{L}^2$ and hence converges a.s. We conclude for $(\Xi_n)_n$ with similar computations, and (2) follows immediately.
\subsubsection{Proof of Proposition \ref{prop:setofequ}}
\label{s:setofequ}
  Let us first prove the following auxilliary Lemma \ref{2epsilon}.
  \begin{lem} \label{2epsilon}
Let $\eps>0$, $i$, $j$ $\in\V$, $i \sim j$, and assume $T_n\ge\Cst(a,\e)$. If $~{x_n\in U_{ij}(\eps)}$ and $|H(x_{n+1})-H(x_n)|<\eps/2$, then $x_{n+1}\in U_{ij}(2\eps)$, where $U_{ij}(\cdot)$ is defined in \eqref{def_Uij}.
\end{lem}

\begin{proof}
Assume $x_n\in U_{ij}(\eps)$. Then either $x_{ij}^n\le \eps$ or $y_{ij}^n-H(x_n)\ge -\eps$.

If  $x_{ij}^n\le \eps$ then, using \eqref{uber}--\eqref{eup}, we deduce $|x_{ij}^{n+1}-x_{ij}^n|\le\Cst(a)/T_n\le\e$ if $T_n\ge\Cst(a,\e)$, so that $x_{ij}^{n+1}\le 2\eps$.

If $x_{ij}^n>\eps$ and $y_{ij}^n-H(x_n)\ge -\eps$, then $|y_{ij}^{n+1}-y_{ij}^n|\le\Cst(a)/(\e^2T_n)\le\e/2$ if $T_n\ge\Cst(a,\e)$ by Lemma \ref{stochapp2}. By assumption, $|H(x_{n+1})-H(x_n)|<\eps/2$, so that we conclude that $y_{ij}^{n+1}-H(x_{n+1})\ge -2\eps$.
\end{proof}

  Let us now show Proposition \ref{prop:setofequ}. We fix $\eps>0$ and $m_0\in\mathbb{N}$, and assume $T_{m_0}\ge \Cst(a_{ij}p_{ij},a,\e)$, so that the assumptions of Lemmas \ref{lem53} and \ref{2epsilon} hold. Let $\tau_{m_0}$
  be the stopping time
 $$\tau_{m_0}=\inf\Big\{n\ge m_0\,:\, x_n\in
  \Delta_{\eps^{4}}  \textrm{ or
  } |H(x_n)-H(x_{m_0})|>\frac{\eps}{4}\Big\}.$$
We want to prove that either $\tau_{m_0}<\infty$, or $x_n\in
  U_{ij}(3\eps)$ for all large $n$. Recall that $U_{ij}(\cdot)$ is defined in~\eqref{def_Uij}. This will allow us to conclude, as Proposition~\ref{prop_p0} 
  and Theorem~\ref{th_pot} imply that there exists almost surely
  $m_0\in\mathbb{N}$ s.t. $\tau_{m_0}=\infty.$

Let $\sigma_{m_0}$ be the stopping time
$$\sigma_{m_0}:=\inf\{n\ge m_0 \,:\, x_n\in U_{ij}(\eps)\}.$$
Lemma~\ref{lem53}~$(2)$ implies that there exists a.s.~a
(random) $m_0\in\mathbb{N}$ such that, for all $n\ge m \ge m_0$,
\beq
 (Y_n-Y_m)^-\le \frac{\eps}{2},\qquad(S_n-S_m)^+\le
\frac{\varepsilon}{2}.\label{RSest}
\eeq
Therefore  $\sigma_{m_0}<\infty$: indeed, otherwise $y_{ij}^n\to\iy$ as$n\to\iy$ by Lemma \ref{lem53} (recall $\tau_{m_0}=\iy$, since $\sum_{n\ge m_0}T_n^{-1}=\infty$, so that $x_n\in U_{ij}^\e$ for large $n$, which leads to a contradiction.

For all $n\ge\sigma_{m_0}$, let $\rho_n$
be the largest $k\le n$ such that $x_k\in U_{ij}(\eps)$. By
(\ref{RSest}), $$(y_{ij}^n-y_{ij}^{\rho_n+1})\ge
-\frac{\eps}{2}.$$
By Lemma \ref{2epsilon}, $x_{\rho_n+1}\in U_{ij}(2\eps)$ : let us assume for instance that $y_{ij}^{\rho_n+1}-H(x_{\rho_n+1})\ge-2\eps$. Together with
$|H(x_n)-H(x_{\rho_n+1})|\le \varepsilon/2$, as $\tau_{m_0}=\iy$, we
deduce that
\beqq
y_{ij}^n-H(x_n)\ge y_{ij}^{\rho_n+1}-H(x_{\rho_n+1})-\eps\ge -3\eps.
\eeqq
With a similar argument, $x_{ij}^{\rho_n+1}\le 2\eps$ implies
$x_{ij}^n\le 3\eps$. These two arguments together imply that $x_n\in U_{ij}(3\eps)$ if
$n\ge\sigma_{m_0}$, which enables us to conclude.

\section{Asymptotic behavior near the boundary}\label{sect_bord}


The results presented in Sections \ref{sect_eq} and \ref{sect_p3} concern respectively the characterisation of stable equilibria and convergence towards them on  $\De\setminus\partial\De$.

However, as we mentioned at the end of Section \ref{sect_th}, a full analysis of the asymptotics of the process $(x_n)_{n\in\N}$ would require a better understanding of the behavior of the dynamics close to the boundary $\partial\De$. 
This can be understood through the following simple example. 

Let $G$ be a star-shaped connected graph consisting of $k+1$ vertices $\V=\{v,u_1,...,u_k\}$, where $v$ is the core adjacent to all vertices $u_l$, $l=1\ldots k$, and any two vertices $u_l$ and $u_m$, $l\ne m$, are not adjacent. Assume w.l.o.g.~that the sequence $a_{vu_l}$, $l=1\ldots k$, is nonincreasing in $i$, and consider the case where $p_\V=1$.

Consider our process $((V_{ij}^n)_{i,j\in\V})_{n\in\N}$ on $G$. Then, for all $l=1\ldots k$, $V_{u_l,v}=V_{u_l}$, so that the probability to choose edge $\{u_l,v\}$ at time $n$ is $V_{u_l,v}^n/V_v^n$ and, in that case, $V_{u_l,v}^{n+1}=V_{u_l,v}^n+a_{vu_l}$. 

This corresponds to the so-called Friedman's urn with $k$ colors: when a ball of color $l$ is chosen, then it is put back into the urn, along with $a_{vu_l}$ other balls of the same color. Thus $V_{u_l,v}^n/n^{a_{vu_l}/a_{vu_1}}$ converges a.s. towards a positive r.v. (see for instance \cite{freedman}).

Assuming for instance that $a_{vu_l}<a_{vu_1}$ for all $l\ne1$, we deduce in particular that $x_{u_2}^n,...,x_{u_k}^n$ converge a.s.~to $0$, and that $x_{u_1}^n,x_v^n$ converge to $1/2$ a.s. Therefore $x_n$ converges to the boundary $\partial\De$.

Note that, for any general graph $G=(\V,E,\sim)$, and for all $i\in\V$ such that $a_{ij}p_{ij}>0$ for some $j\sim i$, the convergence to the boundary will occur at a rate at most $n^{\e-1}$ for some $\e>0$. In particular, this implies that any vertex $i\in\V$ such that $a_{ij}p_{ij}>0$, for some $j\sim i$, performs infinitely many successful communications almost surely. This is a consequence of \eqref{lbni} in Corollary \ref{cor_Tn}: the probability that $i$ receives a positive payoff at time $n+1$ will be at least 
$$x_{i}^nN_i(x_n)=x_{i}^n\sum_{j\in\V}a_{ij}p_{ij}\frac{x_{ij}^2}{(x_{i}^n)^2x_{j}^n}\ge h_{\min}w_{\min}x_i^n=h_{\min}w_{\min}\frac{V_i^n}{T_n}.$$
If we let $L_i^n=h_{\min}V_i^n$, then we can couple the process $(L_i^n)_{n\in\N}$ with a Friedman's urn with $2$ colors ($i$ and $i^c$), so that when a ball of color $i$ is chosen, then it is put back into the urn along with at least $h_{\min}w_{\min}$
other balls of color $i$, and such that in any case (either if $i$ or $i^c$ is chosen), at most $\Cst(a,h_{\min})$ balls of color $i^c$ are added.

Now $F$ is not continuous on the boundary $\partial \De$, so that  Lemma \ref{lem_StochApp} is not useful on the event that the trajectory of $x_n$ has accumulation points in $\partial \De$. However it is possible to write the evolution of our network as a stochastic approximation of solutions of a smooth ODE: if we let $z=((x_i)_{i\in\V}, (y_{ij})_{i,j\in\V})$, then the process $(z_n)_{n\in\N}$ can be seen as a Cauchy Euler approximation of the following ODE: 
\begin{align*}
\dot{x}_i&=x_i(N_i(x)-H(x)),\\
\dot{y}_{ij}&=G(y)_{ij}=y_{ij}(y_{ij}-N_i(x)-N_j(x)+H(x)),\\
N_i(x)&=\sum_{k\in\V}x_ky_{ik}^2.
\end{align*}
However the perturbation in the evolution of $y_{ij}^n$ in Lemma \ref{stochapp2} depends on $x_i$ and $x_j$, and the state space of the ODE on $z$ is not compact anymore.  

\section{Classification of equilibria and stability}\label{sect_eq}


In this section, we analyze the deterministic dynamics associated to the ODE~\eqref{ode}. In particular, we compute the Jacobian matrix of $F$ and give a characterization of the stable equilibria in $\Delta\setminus\partial\Delta$. Let us first do the following remark, which we will use several times.
\begin{rem}\label{remapx}
For any $x\in \Delta$, $H(x)$ is lower bounded by some positive constant, using \eqref{lbHx}. Recalling the definition \eqref{def_Gamma} of $\Gamma$, this implies that, for any $x\in\Gamma$, if  $a_{ij}p_{ij}=0$ then $x_{ij}=0$.
\end{rem}


\subsection{Properties of Lyapunov function} \label{SEC:4:7}


In this section, we show that $H$ is constant on each
connected component of $\Gamma$, defined in \eqref{def_Gamma}.

\begin{prop}
\label{constantH} $H$ is constant on each connected
component of $\Gamma$.
\end{prop}

We first prove in Lemma
\ref{samesupp} that $H$ is constant on connected subsets
of $\Gamma$ with the same support (defined below) by a
differentiability argument. Then we show in Lemma \ref{lem:cont} that $H$ is continuous on $\G$ (including $\Gamma\cap\partial\De$), which enables us to conclude.
Let $\Theta$ be the set of subsets of $E$.
For any $x\in\Delta$, we define its support
\beq\label{def_Ex}
E_x:=\{ij\in  E:
 \, x_{ij}>0\}.
 \eeq
 $\Theta$ can be used as an index set to divide
$\Delta$ or $\Gamma$ into several subsets: for any $\theta
\in \Theta,$
\begin{align*}
  \Delta_{\theta}&:=\{x\in\Delta:\,E_x=\theta\},\\
  \Gamma_{\theta}&:=\Delta_{\theta}\,\cap\,\Gamma.
\end{align*}
\begin{lem}
\label{samesupp}
  For any $\theta\,\in\,\Theta$, $H$ is constant on each connected component of
  $\Gamma_{\theta}$.
\end{lem}
\begin{proof}
Given $q\in\Gamma_\theta$, let us differentiate $H$ at $q$ with respect to $x_{ij}=x_{ji},\,ij\in
\theta$ (i.e. $q_{ij}>0$ and thus $a_{ij}p_{ij}>0$ by Remark \ref{remapx}) without the constraint $x\in\Delta$:
\beqq
\left[\frac{\partial H}{\partial x_{ij}}(x)\right]_{x=q}&=&\left[\sum_{k,l:kl\in \theta} \frac{\partial}{\partial x_{ij}}\bigg( a_{kl}p_{kl} \frac{x_{kl}^2}{x_kx_l}\bigg)\right]_{x=q}\\
&=&\left[2\frac{\partial}{\partial x_{ij}} \bigg(a_{ij}p_{ij}\frac{x_{ij}^2}{x_ix_j}\bigg)+2\sum_{k\ne j, ik\in \theta} \frac{\partial}{\partial x_{ij}} \bigg(a_{ik}p_{ik}\frac{x_{ik}^2}{x_ix_k}\bigg)\right.\\
&&\left.+2\sum_{k\ne i, jk\in \theta} \frac{\partial}{\partial x_{ij}} \bigg(a_{jk}p_{jk}\frac{x_{jk}^2}{x_jx_k}\bigg)\right]_{x=q}\\
&=& 2a_{ij}p_{ij}\frac{q_{ij}}{q_iq_j}\bigg(2 -\frac{q_{ij}}{q_i} -\frac{q_{ij}}{q_j}\bigg)\\
&&-2\sum_{k\ne j:ik\in \theta} a_{ik}p_{ik} \frac{q_{ik}}{q_iq_k}\cdot \frac{q_{ik}}{q_i}-2\sum_{k\ne i:jk\in \theta} a_{jk}p_{jk}\frac{q_{jk}}{q_jq_k}\cdot\frac{q_{jk}}{q_j}\\
&=& 2H(q)\bigg(2 -\frac{q_{ij}}{q_i} -\frac{q_{ij}}{q_j}\bigg)-2H(q)\bigg(1-\frac{q_{ij}}{q_i}\bigg)-2H(q)\bigg(1-\frac{q_{ij}}{q_j}\bigg)\\
&=&0.
\eeqq
The penultimate equality comes from the fact that, for all $r$, $s$ $\in\V$, $rs\in\theta$ (hence $q_{rs},a_{rs}p_{rs}>0$, since $q\in\G_\theta$), we have 
$(a_{rs}p_{rs}q_{rs})/(q_rq_s)= H(q)$.
\end{proof}

\begin{lem} \label{lem:cont} $H$ is continuous on $\Gamma$.
\end{lem}
\begin{proof} Suppose that $q\in \Gamma$,
and  that $x\in \Gamma$ is in the neighbourhood of $q\in \Gamma$ within
$\Delta$, then $E_x \supseteq E_q$ and, using $x\in \Gamma$,
\begin{align*}
 H(x)=\sum_{i,j:ij\in E_q} a_{ij}p_{ij}\frac{x_{ij}^2}{x_ix_j}+\sum_{i,j:ij\in E_x\setminus E_q}x_{ij}H(x),
\end{align*}
so that
$$H(x)=\frac{1}{1-\sum_{i,j:ij\in E_x\setminus E_q} x_{ij}}\sum_{i,j:ij\in E_q}a_{ij}p_{ij} \frac{x_{ij}^2}{x_ix_j},$$
and the conclusion follows.
\end{proof}


\subsection{Jacobian matrix} \label{SEC:4:9}


At any equilibrium $x\in(\Delta\setminus\partial\Delta)\cap \Gamma$ ($H$ and $F$ are not differentiable on $\partial\De$), we
compute the Jacobian matrix
$$J(x)=\bigg(\frac{\partial F_{lk}}{\partial x_{ij}}\bigg)_{\{i,j\}, \{l,k\}:ij,lk\in  E},$$
where, by a slight abuse of notation,
$$F(x)=(F_{ij}(x))_{\{i,j\}:ij\in  E}.$$

For all $ij$, $kl$ $\in  E$, a simple extension of the calculation in the proof of Lemma \ref{samesupp} yields
$\frac{\partial\, H}{\partial x_{ij}}(x)=4H(x)(\1_{\{x_{ij}\neq 0\}}-1)$, so that
\beqq
\frac{\partial F_{lk}}{\partial x_{ij}}(x)&=&-H(x)\1_{\{\{i,j\}=\{l,k\},x_{lk}=0\}}+x_{lk}\frac{\partial y_{lk}}{\partial x_{ij}}(x)- x_{lk} \frac{\partial H}{\partial x_{ij}}(x)\\
&=& -H(x)\1_{\{\{i,j\}=\{l,k\},x_{lk}=0\}}+H(x)\1_{\{\{i,j\}=\{l,k\},x_{lk}\neq 0\}} -\frac{x_{lk}}{x_i}H(x)\1_{\{i\in\{l,k\}\}}\\
&&-\frac{x_{lk}}{x_j}H(x)\1_{\{j\in\{l,k\}\}}-4x_{lk}H(x)\big(\1_{\{x_{ij}\neq 0\}}-1\big).
\eeqq
Recalling Remark \ref{remapx}, note that the last computation holds for $a_{ij}p_{ij}=0$ or $a_{kl}p_{kl}=0$.

Therefore, for $x_{ij}\ne0$, we deduce
\beqq
\frac{\partial F_{ij}}{\partial x_{ij}}&=& H(x)\bigg[ 1-\frac{x_{ij}}{x_i}-\frac{x_{ij}}{x_j}\bigg];\\
\frac{\partial F_{ik}}{\partial x_{ij}}&=& -\frac{x_{ik}}{x_i} H(x), k\ne j;\\
\frac{\partial F_{jk}}{\partial x_{ij}}&=& -\frac{x_{jk}}{x_j}H(x), k\ne i;\\
\frac{\partial F_{lk}}{\partial x_{ij}}&=& 0 , l\ne i,j; k\ne i,j.
\eeqq
For any $ij\in E$ s.t. $x_{ij}=0$, we have:
\beqq
\frac{\partial F_{ij}}{\partial x_{ij}}&=& -H(x)\\
\frac{\partial F_{lk}}{\partial x_{ij}}&=&0 , l\ne i,j; k\ne i,j; x_{lk}=0.
\eeqq

 Let $\mathcal{C}_1, \ldots,\mathcal{C}_d$ be the connected components of  the subgraph $G_x:=(\V,E_x)$, where $E_x$ is defined in~\eqref{def_Ex}.
 Let
 $$J_x^m:=\left(\frac{\partial\,F_{lk}}{\partial
x_{ij}}\right)_{ij,kl\in \mathcal{C}_m}.$$
Therefore, 
$J_x$ can be written as follows, by putting first $ij$ and $kl$ coordinates such that $x_{ij}\ne0$ and $x_{lk}\ne0$ (in the same order, with increasing connected components $\mathcal{C}_1, \ldots,\mathcal{C}_d$)
$$J(x)= \begin{pmatrix}
J_x^1&&  &&&&  & & & \\
 && \ddots && &&& (*) & & \\
&  && &J_x^d &  & &&& \\
&  && & &  & &&&\\
& & & &&&&-H(x)& &\\
& &&& (0)&&& & \ddots& \\
& & &&& & &&& -H(x)
\end{pmatrix}.$$


\subsection{Classification of equilibria based on stability} \label{SEC:4:10}


Let us introduce some definitions on stability for
ordinary differential equations.
\begin{definition}
$x$ is \textnormal{Lyapunov stable} if for any
neighbourhood \textnormal{$U_1$} of $x$, there exists a neighbourhood
\textnormal{$U_2 \subseteq U_1$} of $x$ such that any solution
\textnormal{$x(t)$} starting in \textnormal{$U_2$ } is such that \textnormal{$x(t)$} remains in \textnormal{$U_1$} for all
\textnormal{$t\ge 0$}.
\end{definition}

\begin{definition}
$x$ is \textnormal{asymptotically stable} if it is \textnormal{Lyapunov stable} and there
exists a neighbourhood $U_1$ such that any solution $x(t)$
starting in $U_1$ is such that $x(t)$ converges to $x$.
\end{definition}

An equilibrium that is Lyapunov stable but not asymptotically stable
is   sometimes called $neutrally$ $stable$.

\begin{definition}\label{defstab}
 $x$ is \textnormal{linearly stable}  if all eigenvalues of the Jacobian matrix
 at $x$ have nonpositive real part; otherwise, $x$ is called \textnormal{linearly unstable}.
\end{definition}

Remark that, with these definitions, linear stability allows for eigenvalues to have zero real part, and therefore does not necessarily imply Lyapunov stability. However the dynamics considered here makes these stable equilibria indeed Lyapunov stable: as in~\cite{HST}, as we will observe in Section~\ref{sect_p3} (in a stochastic version).

\begin{definition}\label{def_se}
Let
\begin{align*}
  \Gamma_0&\,:=\,\Gamma\,\cap\,\Delta\setminus\partial\Delta,\\
  \Gamma_b&\,:=\,\Gamma\,\cap\,\partial\Delta,
\end{align*}
and let $\Gamma_s$ (resp. $\Gamma_u$) be the set of linearly stable (resp. unstable) equilibria in
$\Gamma_0$ for the mean-field ODE.
\end{definition}

For any $x \in \Gamma_u$, let
$$\mathcal{E}_x:=\{\theta\in\mathbb{R}^{| E|}\,:\,|\theta|=1 \textnormal{ and } \exists \,ij\in E_x
\textnormal{ s.t. }\theta\cdot \mathbf{e}_{ij}\ne0\},$$
where $E_x$ is defined in~\eqref{def_Ex}.

Recall Definition~\ref{def_Gx} of a graph $G_x$ associated to some $x\in\Delta$.

\begin{prop}
\label{Prop:unstable}
We have
\begin{enumerate}
\item [(a)] $\Gamma_s=\{x\in \Gamma_0 : P_{G_x}\text{ holds}\}$.\\
\item [(b)] If $x\in \Gamma_u$, then there exists an eigenvector in $\mathcal{E}_x$ whose
eigenvalue has positive real part.\\
\end{enumerate}
\end{prop}

We are first going to prove following Lemma \ref{aij=aik} giving a necessary condition on the elements of $\Gamma_s$. Its proof is widely inspired by that of a similar result in \cite{BenTarr}.

\begin{lem} \label{aij=aik}
If $x$ is in $\Gamma_s$, then for all $i,j,k$ and $l$ in the same connected component of $G_x$ such that $i\simx j$, $k\simx l$, we have:
$$a_{ij}p_{ij}=a_{kl}p_{kl}.$$
Let us denote $\mathcal{P}_a$ this last property.
\end{lem}

\begin{proof}
Choose $x\in(\Delta\setminus\partial\Delta)\cap\G$.
It is sufficient to prove that, if $x_{i_0j_0}>0$ and $x_{i_0k_0}>0$, then $a_{i_0j_0}p_{i_0j_0}=a_{i_0k_0}p_{i_0k_0}$. 

Assume the contrary, for some $i_0$, $j_0$, $k_0$. Then, in particular, 
$$a_{i_0j_0}p_{i_0j_0}\frac{x_{i_0j_0}}{x_{i_0}x_{j_0}}=H(x).$$

Recall that   the entries of the Jacobian matrix $J(x)$, computed in Section \ref{SEC:4:9}, are such that, on the line corresponding to $ij$, we have:
\beqq
\frac{\partial F_{ij}}{\partial x_{ij}}&=&H(x)\big[1-\frac{x_{ij}}{x_i}-\frac{x_{ij}}{x_j}\big].\\
\frac{\partial F_{ij}}{\partial x_{ik}}&=&-\frac{x_{ij}}{x_i}H(x), k\ne j,\\
\frac{\partial F_{ij}}{\partial x_{jk}}&=&-\frac{x_{ij}}{x_j}H(x), k\ne i,\\
\frac{\partial F_{ij}}{\partial x_{kl}}&=&0, \{k,l\}\cap\{i, j\}=\emptyset.
\eeqq
Recall also that $J(x)$ is block upper triangular with on its diagonal sub-matrices $J^m_x$, $1\le m \le d$. Hence, the eigenvalues of $J^m_x$, $1\le m \le d$, are eigenvalues of $J(x)$. Moreover, each $J^m_x$ corresponds to the edges within the same connected component of the graph, where we called $\mathcal{C}_m$ this connected component. This means that for any coordinate $kl$ of $J^m_x$, we have $kl\in\mathcal{C}_m$ and $x_{kl}>0$.\\
Let us first prove that the matrices $J^m_x$ can be written as the product of a diagonal matrix and a symmetric matrix. Define the diagonal matrix
\[
D^m:=\left(x_{ij}H(x)\1_{\{ij=kl\}}\right)_{ij,kl\in\mathcal{C}_m}.
\]
Besides, define the symmetric matrix $M^m$, with entries of the coordinates $(ij,kl)$, $ij,kl\in\mathcal{C}_m$,  given by
\beqq
M^m_{ij,ij}&=&\frac{1}{x_{ij}}-\frac{1}{x_i}-\frac{1}{x_j},\\
M^m_{ij,ik}&=&-\frac{1}{x_i},\ k\neq j,\\
M^m_{ij,jk}&=&-\frac{1}{x_j},\ k\neq i,\\
M^m_{ij,kl}&=&0,\ \{i,j\}\cap\{k,l\}=\emptyset.
\eeqq
It is then straightforward to check that $J^m_x=D^mM^m$.\\

Now, the eigenvalues of $J^m_x$ are all nonpositive if and only if the eigenvalues of $M^m$ are all nonpositive. To have a proof of this fact, see the Claim in the proof of Lemma 1 in \cite{BenTarr}, p. 2199. Therefore, let us prove that $M^m$ has a positive eigenvalue, which will enable us to conclude.\\

Let $u$ be a vector in $\mathbb{R}^{|  \mathcal{C}_m |}$, such that:
$$u^T=(0, ..., 0,  u_{i_0j_0}, 0 ..., 0, u_{i_0k_0}, 0, ..., 0).$$
Dropping the $0$ of $i_0$, $j_0$ and $k_0$ for simplicity, we have:
\beqq
\big(M^m\cdot u\big)_{ij}&=& u_{ij}\big[\frac{1}{x_{ij}}-\frac{1}{x_i}-\frac{1}{x_j}\big]-u_{ik}\frac{1}{x_i}\\
\big(M^m\cdot u\big)_{ik}&=& u_{ik}\big[\frac{1}{x_{ik}}-\frac{1}{x_i}-\frac{1}{x_k}\big]-u_{ij}\frac{1}{x_i},
\eeqq
hence
\beqq
u^TM^mu&=& u_{ij}^2\big[\frac{1}{x_{ij}}-\frac{1}{x_i}-\frac{1}{x_j}\big]+u_{ik}^2\big[\frac{1}{x_{ik}}-\frac{1}{x_i}-\frac{1}{x_k}\big]\\
&&-u_{ij}u_{ik}\big(\frac{1}{x_i}+\frac{1}{x_i}\big).
\eeqq

Then, we chose $u_{ij}=1$ and $u_{ik}=-1$. So, we have:
\beqq
u^TM^mu&=& \frac{1}{x_{ij}}\big[1-\frac{x_{ij}}{x_j}\big]+\frac{1}{x_{ik}}\big[1-\frac{x_{ik}}{x_k}\big].
\eeqq
Recall that $\frac{x_{ij}}{x_j}\leq 1$ and $\frac{x_{ik}}{x_k}\leq 1$, and notice that $\frac{x_{ij}}{x_j}\neq 1$ or $\frac{x_{ik}}{x_k}\neq 1$. Indeed, if $\frac{x_{ij}}{x_j}=\frac{x_{ik}}{x_k}=1$, as $a_{ij}p_{ij}\frac{x_{ij}}{x_j}=a_{ik}p_{ik}\frac{x_{ik}}{x_k}$ (we are on the set of equilibria), then we would have $a_{ij}p_{ij}=a_{ik}p_{ik}$, which is not the case, by assumption.\\
Hence $u^TM^mu>0$ and therefore $M^m$ has a positive eigenvalue, as this matrix is symmetric. Finally, this implies that $J(x)$ has a positive eigenvalue and therefore $x$ is not stable.
\end{proof}

Lemma \ref{aij=aik} implies that $\Gamma_s\subseteq\{x\in \Gamma_0:\mathcal{P}_a \text{ holds}\}=:\Gamma_0^a$.\\
To prove Proposition \ref{Prop:unstable}, we need the following  Lemma \ref{Pdnh} on the structure of ${G}_x$ when $x\in \Gamma_0^a$.

\begin{lem} \label{Pdnh} 
For all $x\in\Gamma_0^a$ such that ${P}_{{G}_x}$ does not hold, then there is at least one connected component on which every vertex has at least two edges.
\end{lem}

\begin{proof}
Assume that $x$ is in $\Gamma_0^a$ and ${P}_{G_x}$ does not hold. Assume by contradiction that, on any connected component,  there exists a vertex $j$ linked to only one vertex, say $i$. So, we have $a_{ij}p_{ij}>0$ by Remark \ref{remapx}, and $\frac{x_{ij}}{x_j}=1$. Then, as $x\in\Gamma_0^a$, for all $k$ such that $x_{ik}>0$, we have $a_{ik}p_{ik}=a_{ij}p_{ij}$ and
$$y_{ik}=a_{ik}p_{ik}\frac{x_{ik}}{x_ix_k}=y_{ij}=a_{ij}p_{ij}\frac{x_{ij}}{x_ix_j} \Rightarrow \frac{x_{ik}}{x_k}=1,$$
then $k$ is only linked to $i$, and therefore each connected component of ${G}_x$ is star-shaped, which contradicts the assumption that ${P}_{{G}_x}$ does not hold, and allows us to conclude.
\end{proof}

\noindent\textbf{Proof of Proposition \ref{Prop:unstable}.}
First, assume that $x\in\Gamma_0^a$ and ${P}_{{G}_x}$ does not hold. We want to prove that $x\in\Gamma_u$. Lemma \ref{Pdnh} implies that ${G}_x$ has a connected component on which each vertex is adjacent to at least two edges, which we assume w.l.o.g.~to be $\mathcal{C}_1$. Let $V(\mathcal{C}_1)$ (resp. $E(\mathcal{C}_1)$) be its set of vertices (resp. edges).  Let us show that $J_x^1$ has at least one eigenvalue with positive real part.

Compute the trace of $J_x^1$ :
\beqq
Tr(J_x^1)&=& H(x) \sum_{\{i,j\}:i\simx j} (1-\frac{x_{ij}}{x_i}-\frac{x_{ij}}{x_j})\\
&=&H(x)(|E(\mathcal{C}_1)|-|V(\mathcal{C}_1)|)\ge0.
\eeqq

The last inequality comes from the fact that each vertex has at least two edges.\\
It is easy to check that $(1,...,1)\cdot J_x^1=-H(x)(1,...,1)$, hence $-H(x)$ is an eigenvalue of $J_x^1$ and $-H(x)<0$ by Remark \ref{remapx}. Subsequently, there exists an eigenvalue with positive real part. This implies part ${(a)}$ of Proposition \ref{Prop:unstable}, using Lemma \ref{aij=aik}.

Now, assume that $x\in\Gamma_0^a$ and that ${P}_{{G}_x}$ holds. Then, each connected component of ${G}_x$ is star-shaped. Let us assume for instance that $\mathcal{C}_1$ is composed of a nucleus  $v$ and satellite vertices $1,...,k$. Then, for all $i\in\{1,..,k\}$, $x_{iv}=x_i$, and we have:
$$J_x^1= -\frac{H(x)}{x_v}\begin{pmatrix}
x_1 & \dots &  x_1\\
x_2 & \dots & x_2\\
\vdots & & \vdots\\
x_k & \dots & x_k
\end{pmatrix}.$$
The rank of $J_x^1$ is $1$, $0$ and $-H(x)<0$ are eigenvalues, as $(1,...,1)\cdot J_x^1=-H(x)(1,...,1)$, so all its eigenvalues are nonpositive, which completes the proof.
\begin{flushright}$\square$\end{flushright}

\begin{rem}Note that, in general, there may be no stable configuration on $\Delta\setminus\partial\Delta$.
Now, if the number of vertices is even, every equilibrium  $x$ such that $G_x$ yields a one-to-one correspondence between vertices (i.e. each connected component has only two vertices and one edge)  is asymptotically stable.
\end{rem}

Finally, the following proposition implies the second part of Proposition \ref{Prop:unstable0}. Recall Definition \ref{def_Gx} of $G_x$ for $x\in\Delta$, Definition \ref{def_prop} of the property $P_\Gc$ and Definition \ref{def_gammag} of the set $\Gamma_\Gc$, for some subgraph $\Gc$. Recall also the Definition \ref{def_se} of the set $\Gamma_s$ of stable equilibria outside the boundary.

\begin{prop}\label{prop:eqg}
Let $\Gc\subset G$ be a subgraph such that $P_\Gc$ holds and let $q\in\Delta$ be such that $G_q=\Gc$.  Then, $q$ is an equilibrium if and only if $q\in\Gamma_\Gc$ and, in this case, $q$ is stable. In other words, $\Gamma_s\cap \{q\in\Delta:G_q=\Gc\}=\Gamma\cap \{q\in\Delta:G_q=\Gc\}=\Gamma_\Gc$.
\end{prop}
\begin{proof}
Let $q\in\Delta$ be such that $G_q=\Gc$. As $P_\Gc$ holds, $q\in\Delta\setminus\partial\Delta$ and if $a_{ij}p_{ij}=0$ then $q_{ij}=0$. Besides, each connected component of $\Gc$ is star-shaped and contains a single nucleus vertex (see Definition \ref{def_nucleus}). We denote $N_\Gc$ the set of nucleus vertices of $\Gc$.\\
Now, for any $i\in N_\Gc$, we have that $a_{ij}p_{ij}=a_{ik}p_{ik}=:(ap)_i$ as soon as $q_{ij}q_{ik}>0$. If $i\in N_\Gc$ is such that $a_{ij}p_{ij}=0$ for any $j\in\V$, then $i$ is isolated in $\Gc$ and we let $(ap)_i:=0$. Moreover,  for any $i\in N_\Gc$ and $j\in\V$ with   $q_{ij}>0$, we have $q_{ij}/q_j=1$. Using this, a simple computation yields
\[
H(q)=2\sum_{i\in N_\Gc} (ap)_i,
\]
and if $q$ is an equilibrium, then, for any $i\in N_\Gc$, $q_i=(ap)_i/(2\sum_{j\in N_\Gc} (ap)_j)$, hence $q\in\Gamma_\Gc$ and $q$ is stable by Proposition \ref{Prop:unstable}.\\
Conversely, it is straightforward to check that if $q\in\Gamma_\Gc$ then $q\in\Delta\setminus\partial\Delta$,  $G_q=\Gc$ and $q\in\Gamma$, hence $q$ is a stable equilibrium by Proposition \ref{Prop:unstable}.
\end{proof}


\section{Proof of Theorem~\ref{th_posprob}}
\label{sect_p3}


Assume that $(G,(a_{ij},p_{ij})_{ij\in E})$ is such that $\G_s\neq\emptyset$, where $\Gamma_s$ is the set of stable equilibria in $\Delta\setminus\partial\Delta$ (see Definition \ref{def_se}). 

Theorem \ref{th_posprob} is a consequence of the following Proposition \ref{prop:posconv}. Before stating it, we need some definitions.

Let $\mathcal{G}\subseteq G$ be a graph such that ${P}_{\mathcal{G}}$ holds (see Definition~\ref{def_prop}) and let $E_{\mathcal{G}}$ be  the set of edges of $\mathcal{G}$. Besides, let us denote $V_{\mathcal{G}}$ the set of non-isolated vertices of $\mathcal{G}$ (in particular, as ${P}_{\mathcal{G}}$ holds, $i\in V_\mathcal{G}$ if and only if $a_{ij}p_{ij}>0$ for some $j\sim i$). 
Each connected component contains a single nucleus vertex $i$ (see Definition \ref{def_nucleus}) and is such that $a_{ij}p_{ij}=a_{ik}p_{ik}>0$ as soon as $ij, ik\in E_{\mathcal{G}}$. 
Finally, we denote by $N_\mathcal{G}$ the set of nucleus vertices of $\mathcal{G}$.\\
Let $\pi:=\pi_\mathcal{G} : \V\longrightarrow N_\mathcal{G}$ be a function mapping $i\in \V$ to the nucleus vertex of the connected component that $i$ belongs to. In particular, note that, for any $i\in\V\setminus V_{\mathcal{G}}$, $\pi(i)=i$ and $\pi(j)\neq i$ for any $j\neq i$.

 For all $i\in \V$, $n\in\N$ and $\eps>0$, let us define
\begin{align*}
\alpha_{i}^n&:=x_{i}^n/x_{\pi(i)}^n,\\
H_n^{1}&:=\bigcap_{i\in V_{\mathcal{G}},i=\pi(i)}\{\,
    \,V_{i}^n\ge 2\varepsilon n\}, \\
 H_n^{2}&:=\,\,\,\,\,\,\bigcap_{i\in\V}\,\,\,
\{\alpha_{i}^n\ge\varepsilon\}, \\
    H_n^{3}&:=\bigcap_{i,j\in \V, i\ne\pi(j),j\ne\pi(i)}\{\,V_{ij}^n\le \sqrt{n}\,\}.
  \end{align*}
 Obviously, all these definitions depend on the graph $\mathcal{G}$, but for simplicity we do not write this dependency.

\begin{prop}
\label{prop:posconv}
 Let $\mathcal{G}$ be such that $\mathcal{P}_\mathcal{G}$ holds, and let $\pi:=\pi_\mathcal{G}$. For all $\varepsilon\in(0,w_{\min})$, if $H_n^{1}$, $H_n^{2}$ and $H_n^{3}$ hold, and $n\ge\verb?Cst?(\eps,(p_{ij},a_{ij},v_{ij}^0)_{i,j\in \V},|\V|)$, then, with lower bounded probability (only depending on $\eps$, $(p_{ij},a_{ij},v_{ij}^0)_{i,j\in \V}$, and $|\V|$), for all $i$, $j$ $\in \V$, $k\ge n$,
  \begin{align} \label{verve1}
    &V_{ij}^{\infty}=V_{ij}^n,\textrm{ when } i\ne\pi(j),j\ne\pi(i);\\ \label{verve2}
    &\alpha_{i}^k/\alpha_{i}^n\to \alpha_i\in (1-\varepsilon,1+\varepsilon);\\ \label{verve3}
    &V_{i}^k\ge \varepsilon k, \textrm{ when }i\in V_{\mathcal{G}}\text{ and } \pi(i)=i.
  \end{align}
\end{prop}
Before proving this proposition, let us explain how it implies Theorem \ref{th_posprob}. 
Firstly, by Proposition \ref{Prop:unstable0}, if $q$ is a stable equilibrium with $G_q=\Gc$, then $q\in\Gamma_\Gc$, where $\Gamma_\Gc$ is defined in Definition \ref{def_gammag}. Secondly, choose such a stable equilibrium $q\in\Gamma_\Gc$ and let us prove that $(x_n)$ converges to an equilibrium in a neighborhood of $q$ with positive probability.

Let us work on a probability event $A$ on which $H_{n_0}^{1}$, $H_{n_0}^{2}$ and $H_{n_0}^{3}$ hold for some $n_0\ge\verb?Cst?(\eps,(p_{ij},a_{ij},v_{ij}^0)_{i,j\in \V},|\V|)$, and \eqref{verve1}, \eqref{verve2} and \eqref{verve3} occur. Moreover, assume that, on $A$, $\alpha_i^{n_0}\in ((1-\e)q_i/q_{\pi(i)},(1+\e)q_i/q_{\pi(i)})$ for any $i\in\V$.

Note that \eqref{verve1} implies that, for any $i\in \V\setminus V_{\Gc}$, $V_i^\infty<\infty$ and thus $x_i^n\to 0$. Also, for $i\in V_\Gc\cap N_\Gc$ and $j\in \V$ such that $i\neq \pi(j)$, we have that $V_{ij}^\infty<\infty$ and $x_{ij}^n\to0$, using \eqref{verve1}. Besides, for any $i\in V_\Gc\cap N_\Gc$, $\liminf_n x_i^n\ge \e$ by \eqref{verve3}, $x_{ij}^n/x_j^n\to1$ for any $j$ such that $i=\pi(j)$ by \eqref{verve1} and $x_{ij}^n/x_i^n$ converges to some positive limit by \eqref{verve2}. Therefore, as $x_n$ stays away from $\partial\Delta$, it easily implies that $H(x_n)\to 2\sum_{i\in N_\Gc} a_ip_i$.   Additionally, as $(x_n)$ converge almost surely to $\Gamma$, we necessarily have that $x_i^n\to (ap)_i/(2\sum_{j\in N_\Gc} (ap)_j)=q_i$ for any $i\in V_\Gc\cap N_\Gc$ and where $(ap)_i:=a_{ik}p_{ik}$ for some/any $k$ such that $q_{ik}>0$.

Together with \eqref{verve2}, this implies that, on $A$, $(x_n)$ converges to some $x\in\Delta$ in a neighborhood of $q$, and $A$ has positive probability as soon as $n_0$ is large enough. Moreover, the limit $x$ is such that $G_x=\Gc$ and, for any $i\in V_\Gc\cap N_\Gc$, $x_i= (ap)_i/(2\sum_{j\in N_\Gc} (ap)_j)$ and $x_{ij}>0$ iff $i=\pi(j)$. Consequently, $x\in\Gamma_\Gc$ and is thus a stable equilibrium by Proposition \ref{Prop:unstable0}.

In the remainder of this section, we fix the graph $(\mathcal{G},\stackrel{g}{\sim})$ (and thus $\pi=\pi_\mathcal{G}$ and the events $H_n$'s) and $\varepsilon>0$. The proof of Proposition \ref{prop:posconv} consists of the following Lemmas \ref{claim4}, \ref{claim5} and \ref{claim6}.

  Let, for all $i$, $j$ $\in \mathbb{V}$, $n\in\N$,
  \begin{align*}
    \tau_n^{1,i,j}&:=\inf\{k\ge n\,:\, V_{ij}^k\ne V_{ij}^n\};\\
    \tau_n^{2,i}&:=\inf\{k\ge n\,:\,
    \alpha_i^k/\alpha_i^n\notin (1-\varepsilon,1+\varepsilon)\};\\
    \tau_n^{3,i}&:=\inf\{k\ge n\,:\, V_{i}^k<{\varepsilon}
    k\},
\end{align*}
and let
$$\tau_{n}^1:=\inf_{i,j\in \V, i\ne\pi(j), j\ne\pi(i)}\tau_n^{1,i,j},\,\,\,
\tau_{n}^2:=\inf_{i\in \V}\tau_n^{2,i},\,\,\,
\tau_{n}^3:=\inf_{i\in V_{\mathcal{G}}, \pi(i)=i}\tau_n^{3,i},\,\,\,\tau_n:=\tau_n^1\wedge\tau_n^2\wedge\tau_n^3.$$

\begin{lem}
  \label{claim4}
 If  $n\ge\verb?Cst?(\eps,(p_{ij},a_{ij})_{i,j\in \V},|\V|)$, then
  \begin{align*}
  \mathbb{P}\Big(\, \tau_n^1>\tau_n^2\land \tau_n^3\,|\,\mathcal{F}_n,
  H_n^1,H_n^2,H_n^3\, \Big)
  \,\ge\,\exp\left(-4|\V|^2\max_{i,j\in\V}\{a_{ij}\}\varepsilon^{-4}\right).
  \end{align*}
\end{lem}
\begin{proof}
Assuming $n\ge\verb?Cst?(\eps,(p_{ij},a_{ij},v_{ij}^0)_{i,j\in \V},|\V|)$,
\begin{align*}
  \mathbb{P}\Big(\,\tau_n^1>\tau_n^2\land
  \tau_n^3\,|\,\mathcal{F}_n,H_n^1,
  H_n^2,H_n^3\Big)&\ge\prod_{k\ge
  n}
  \left(1-\sum_{i,j:i\ne\pi(j),j\ne\pi(i)} a_{ij}p_{ij}\frac{(V_{ij}^n)^2}{V_i^kV_j^k}\right)\\
&\ge\exp\left(-\frac{3\max_{i,j\in\V}\{a_{ij}\}}{2}\sum_{i,j:i\ne\pi(j),j\ne\pi(i), k\ge n}\frac{n}{\eps^4k^2}\right)\\ 
&\ge\exp\left(-4|\V|^2\max_{i,j\in\V}\{a_{ij}\}\varepsilon^{-4}\right).
\end{align*}
\end{proof}

\begin{lem}
  \label{claim5}
 If  $n\ge\verb?Cst?(\eps,(p_{ij},a_{ij},v_{ij}^0)_{i,j\in \V},|\V|)$ then, for all $i\in \mathbb{V}$,
  \beqq
    \mathbb{P}\Big(\,\tau_n^{2,i}>\tau_n^1\land \tau_n^3\,|\,\mathcal{F}_n,
  H_n^1,H_n^2,H_n^3\,\Big)\ge1-2\exp(-\verb?Cst?(\varepsilon)n).
  \eeqq
  Moreover, given that $  H_n^1$, $H_n^2$ and $H_n^3$ hold, $\alpha^i_{k\wedge \tau_n}/\alpha^i_n$ converges almost surely as $k$ goes to infinity.
\end{lem}
\begin{proof}
Notice that if $i\in \V$ is such that $i=\pi(i)$, then $\alpha_i^n=1$ a.s.~for any $n\in\mathbb{N}$ and consequently $\tau_n^{2,i}=\infty$ almost surely. Then, fix $i\in \V$, $n\in\N$, and assume that $\pi(i)\ne i$ (in particular $i\in V_{\mathcal{G}}$).

Let, for all $j\in \V$ and $k\ge n$,
$$\hat{V}_j^k:={V}_j^n+\sum_{l\stackrel{g}{\sim} j}(V_{jl}^k-V_{jl}^n),$$
which is equal to $V_j^k$ as long as $k<\tau_n^1$.

Let, for all $k\ge n$,
$$W_k:=\log\frac{\hat{V}_i^k}{\hat{V}_{\pi(i)}^k},$$
and let us consider the Doob decomposition of $(W_k)_{k\ge n}$:
\begin{align*}
W_k&=W_n+\De_k+\Psi_k,\\
\De_k&:=\sum_{j= n+1}^k\E(W_j-W_{j-1}|\mathcal{F}_{j-1}).
\end{align*}

In the following computation, we write $u=\Box(v)$ if $|u|\le v$, for all $u,v\in\mathbb{R}$.\\
Assume that $H_n^1$, $H_n^2$ and $H_n^3$ hold, and that $k<\tau_n$: then, using that, for all $j\stackrel{g}{\sim} \pi(i)$, $p_{i\pi(i)}a_{i\pi(i)}=p_{j\pi(i)}a_{j\pi(i)}$,
\beqq
&&|\De_{k+1}-\De_k|=
  \left|\,\mathbb{E}\left[\,W_{k+1}-W_k\,|\,\mathcal{F}_k\right]\,\right|\,\\
  &&=\,\left|a_{i\pi(i)}p_{i\pi(i)}
  \left(\frac{V_{i\pi(i)}^k}{V_i^k}\right)^2\frac{1}{V_{\pi(i)}^k}(1+\Box((V_i^k)^{-1})\right.\\
&&\qquad\left. -\sum_{j\stackrel{g}{\sim}\pi(i)}a_{j\pi(i)}p_{j\pi(i)}
  \frac{V_{\pi(i)j}^k}{V_j^k}\frac{V_{\pi(i)j}^k}{(V_{\pi(i)}^k)^2}(1+\Box((V_{\pi(i)}^k)^{-1})\right|\\
&&=\frac{a_{i\pi(i)}p_{i\pi(i)}}{V_{\pi(i)}^k}\Bigg|1+k^{-1/2}\Box(\verb?Cst?(\eps,(p_{ij},a_{ij})_{i,j\in \V},|\V|))\\
&&\qquad\left.-\sum_{j\stackrel{g}{\sim}\pi(i)}\left(1+k^{-1/2}\Box(\verb?Cst?(\eps,(p_{ij},a_{ij})_{i,j\in \V},|\V|))\right)
\frac{V_{\pi(i)j}^k}{V_{\pi(i)}^k}\right|\\
&&=k^{-3/2}\Box\left(\verb?Cst?(\eps,(p_{ij},a_{ij})_{i,j\in \V},|\V|)\right),
\eeqq
where we use that, for all $j\stackrel{g}{\sim}\pi(i)$,
\begin{equation}
\label{estsm}
\left|\frac{V_{\pi(i)j}^k}{V_j^k}-1\right|,
\left|\sum_{j\stackrel{g}{\sim}\pi(i)} \frac{V_{\pi(i)j}^k}{V_{\pi(i)}^k}-1\right| \le k^{-1/2}\verb?Cst?(\eps,(p_{ij},a_{ij},v_{ij}^0)_{i,j\in \V},|\V|).
\end{equation}

Therefore, for all $k\ge n$,
$$|\De_k|\le n^{-1/2}\verb?Cst?(\eps,(p_{ij},a_{ij},v_{ij}^0)_{i,j\in \V},|\V|),$$
which is less than $\e/4$ as soon as $n$ is large enough. Moreover $\De_{k\wedge\tau_n}$ converges a.s.

Let us now estimate the martingale increment: $|\Psi_{k+1}-\Psi_k|\le\verb?Cst?(\eps)k^{-1}$ (since
$|W_{k+1}-W_k|\le\verb?Cst?(\eps)k^{-1}$), so that \cite[Lemma~7.4]{HST} (stated here as Lemma~\ref{martest}) implies
$$\Pb\left(\sup_{k\ge n}|\Psi_{k\wedge\tau_n}-\Psi_n|\le\eps/4\right)\ge1-2\exp(-\verb?Cst?(\varepsilon)n).$$
Moreover, the martingale $(\Psi_{k\wedge\tau_n})_{k\ge n}$ is bounded hence a.s.~converges. This completes the proof.
\end{proof}

Recall the definition \eqref{singe} of $w_{\min}$.

\begin{lem}
  \label{claim6}
If $\eps\in(0,w_{\min})$ and $n\ge\verb?Cst?(\eps,(p_{ij},a_{ij},v_{ij}^0)_{i,j\in \V},|\V|)$ then, for all $i\in V_{\mathcal{G}}$ such that $\pi(i)=i$ (i.e. $i\in N_\mathcal{G}$),
\beqq
  \mathbb{P}\Big(\,\tau_n^{3,i}>\tau_n^1\land \tau_n^2\,\big|\,
  \mathcal{F}_n,H_n^1,H_n^2,H_n^3\,\Big)>1-2\exp(-\verb?Cst?(\eps)n).
\eeqq
\end{lem}
\begin{proof}
Let $n\in\N$, assume that $H_n^1$, $H_n^2$ and $H_n^3$ hold, and fix $i\in V_{\mathcal{G}}$ such that $\pi(i)=i$. Let us consider the Doob decomposition of $(V_i^k)_{k\ge n}$:
\begin{align*}
V_i^k&:=V_i^n+\Phi_k+\Xi_k\\
\Phi_k&:=\sum_{j= n+1}^k\E\left(V_i^j-V_i^{j-1}\,|\,\mathcal{F}_{j-1}\right).
\end{align*}

Now, for all $\eta>0$, if $n\ge\verb?Cst?(\eta,\eps)$ and $k<\tau_n$, \eqref{estsm} implies
$$\Phi_{k+1}-\Phi_k= \E\left(V_i^{k+1}-V_i^{k}\,|\,\mathcal{F}_{k}\right)
\ge\sum_{j\stackrel{g}{\sim} i}a_{ij}p_{ij}\frac{(V_{ij}^k)^2}{V_i^kV_j^k}\ge w_{\min}-\eta,$$
if $n\ge\Cst(\e,\eta,p,a,|\V|)$.\\

Let us now estimate the martingale increment: let, for all $p\ge n$,
$$\chi_p:=\sum_{k=n}^{p-1}\frac{\Xi_{k+1}-\Xi_k}{k}.$$
Then, for all $p\ge n$,
$$\Xi_p=\sum_{n\le k\le p-1}(\chi_{k+1}-\chi_k)k
=-\sum_{n\le k\le p-1}\chi_k+(p-1)\chi_p.$$

This implies, using \cite[Lemma~7.4]{HST} (see Lemma~\ref{martest}) and $|\Xi_{k+1}-\Xi_k|\le1$ for all $k\ge n$, that for all $\eps>0$
\beqq
&&\Pb\left(\forall k\ge n,\,\, V_i^k\ge(2\eps-\eta)n+(k-n)(w_{\min}-\eta)\,|\,\mathcal{F}_n\right)\\ 
&&\ge
\Pb\left(\sup_{p\ge n}\left|\frac{\Xi_p}{p}\right|\le\eta\,|\,\mathcal{F}_n\right)
\ge\Pb\left(\sup_{k\ge n}\left|\chi_k\right|\le\frac{\eta}{2}\,|\,\mathcal{F}_n\right)
\ge1-2\exp(-\verb?Cst?(\eta)n);
\eeqq
we choose $\eta=\min(\eps,w_{\min}-\eps)$, which completes the proof.

\end{proof}

The following Lemma \ref{martest} states an exponential inequality for martingales (see for instance \cite{HST}, Lemma~7.4 for a proof).
\begin{lem}
\label{martest}
Let $(\gamma_k)_{k\in\N}$ be a deterministic sequence of positive reals, let $\mathbb{G}:=(\mathcal{G}_n)_{n\in\N}$ be a filtration, and let $(M_n)_{n\in\N}$ be a $\mathbb{G}$-adapted martingale such that $|M_{n+1}-M_n|\le\gamma_n$ for all $n\in\N$. Then, for all $n\in\N$ and $\lambda>0$,
$$\Pb\left(\sup_{k\ge n}(M_k-M_n)\ge\lambda\,|\,\mathcal{G}_n\right)\le\exp\left(-\frac{\lambda^2}{2\sum_{k\ge n}\gamma_k^2}\right).$$
\end{lem}


\nocite{*}
\bibliographystyle{plain}
\bibliography{bib-net}

\begin{thebibliography}{10}

\bibitem{APSV}
R.~Argiento, R.~Pemantle, B.~Skyrms, and S.~Volkov.
\newblock Learning to signal : Analysis of a micro-level reinforcement model.
\newblock {\em Stochastic processes and their applications}, 119(2):373--390,
  2009.

\bibitem{beggs}
A.~W. Beggs.
\newblock On the convergence of reinforcement learning.
\newblock {\em Journal of Economic Theory}, 122:1--36, 2005.

\bibitem{Ben99}
M.~Bena\"im.
\newblock {\em Dynamics of stochastic approximation algorithms}.
\newblock In Seminaires de Probabilit{\'e}s XXXIII, volume 1709 of Lecture
  notes in mathematics. Springer-Verlag, 1999.

\bibitem{BenTarr}
M.~Bena{\"\i}m and P.~Tarr{\`e}s.
\newblock Dynamics of vertex-reinforced random walk.
\newblock {\em Ann. Probab.}, 39(6):2178--2223, 2011.

\bibitem{DCD2}
Didier Dacunha-Castelle and Marie Duflo.
\newblock {\em Probability and statistics. {V}ol. {II}}.
\newblock Springer-Verlag, New York, 1986.
\newblock Translated from the French by David McHale.

\bibitem{Dur04}
R.~Durrett.
\newblock {\em Probability: Theory and Examples}.
\newblock Duxbury Press, Belmont, CA, Third Edition, 2004.

\bibitem{ER98}
I.~Erev and A.~E. Roth.
\newblock Predicting how people play games: Reinforcement learning in
  experimental games with unique, mixed strategy equilibria.
\newblock {\em The American Economic Review}, 88:848--881, 1998.

\bibitem{freedman}
David~A. Freedman.
\newblock Bernard {F}riedman's urn.
\newblock {\em Ann. Math. Statist}, 36:956--970, 1965.

\bibitem{FL98}
D.~Fudenberg and D.~K. Levine.
\newblock {\em The Theory of Learning in Games}.
\newblock Cambridge: MIT Press, 1998.

\bibitem{HS98}
J.~Hofbauer and K.~Sigmund.
\newblock {\em Evolutionary Games and Population Dynamics}.
\newblock Cambridge: Cambridge University Press, 1998.

\bibitem{hopkins-posch}
E.~Hopkins and M.~Posch.
\newblock Attainability of boundary points under reinforcement learning.
\newblock {\em Games and Economic Behavior}, 53:110--125, 2005.

\bibitem{HST}
Yilei Hu, Brian Skyrms, and Pierre Tarr\`es.
\newblock Reinforcement learning in signaling game.
\newblock {\em Preprint}, 2011.

\bibitem{Lew69}
D.~Lewis.
\newblock {\em Convention: A Philosophical Study}.
\newblock Harvard: Harvard University Press, 1969.

\bibitem{May82}
J.~{Maynard Smith}.
\newblock {\em Evolution and the Theory of Games}.
\newblock Cambridge: Cambridge University Press, 1982.

\bibitem{Pem88}
R.~Pemantle.
\newblock {\em Random processes with reinforcement}.
\newblock Massachussets Institute of Technology doctoral dissertation, 1988.

\bibitem{Pem90}
R.~Pemantle.
\newblock Nonconvergence to unstable points in urn models and stochastic
  approximations.
\newblock {\em Annals of Probability}, 18:698--712, 1990.

\bibitem{Sky10}
B.~Skyrms.
\newblock {\em Signals: Evolution, Learning, and Information}.
\newblock Oxford: Oxford University Press, 2010.

\bibitem{RPBS}
Brian Skyrms and Robin Pemantle.
\newblock A dynamic model of social network formation.
\newblock {\em Proceedings of the National Academy of Sciences},
  97(16):9340--9346, 2000.

\bibitem{Tar00}
P.~Tarr{\`e}s.
\newblock Pi{\`e}ges r{\'e}pulsifs.
\newblock {\em C.R.Acad.Sci.Paris S{\'e}r. I Math}, 330:125--130, 2000.

\bibitem{TarDp}
P.~Tarr{\`e}s.
\newblock Bandit {\`a} deux bras.
\newblock {\em Traps of stochastic algorithms and vertex-reinforced random
  walks (Dphil)}, pages 59--65, 2001.

\bibitem{Tar04}
P.~Tarr{\`e}s.
\newblock Vertex-reinforced random walk on {$\Bbb Z$} eventually gets stuck on
  five points.
\newblock {\em Ann. Probab.}, 32(3B):2650--2701, 2004.

\bibitem{young:2005}
Peyton Young.
\newblock {\em Strategic learning and its limits}.
\newblock Oxford: Oxford University Press, 2004.

\end{thebibliography}

\end{document}